\documentclass[sn-mathphys]{sn-jnl} 

\jyear{2022}%

\theoremstyle{thmstyleone}%
\newtheorem{thm}{Theorem}
\newtheorem{prop}[thm]{Proposition}

\newtheorem{cor}[thm]{Corollary}
\newtheorem{lemma}[thm]{Lemma}

\theoremstyle{thmstylethree}
\newtheorem{defn}[thm]{Definition}
\newtheorem{nota}[thm]{Notation}

\newtheorem{example}[thm]{Example}

\def\V{{\mathcal V}}
\begin{document}

\title[Triangulations of Grassmannians and flag manifolds]{Triangulations of Grassmannians and flag manifolds}


\author{\fnm{Olakunle S} \sur{Abawonse}}\email{o.abawonse@northeastern.edu}

\affil{\orgdiv{Department of Mathematics}, \orgname{Northeastern University}, \\ \orgaddress{\street{360 Huntington Ave}, \city{Boston}, \postcode{02115}, \state{Massachusetts}, \country{USA}}}


\abstract{MacPherson in ~\cite{macp:rob} conjectured that the Grassmannian
$\mathrm{Gr}(2, \mathbb{R}^n)$ has the same homeomorphism type as the combinatorial Grassmannian $\mbox{MacP}(2,n)$, while Babson ~\cite{bab:eric} proved that the spaces $\mbox{Gr}(2,\mathbb{R}^n)$ and $\mbox{Gr}(1,2,\mathbb{R}^n)$ are homotopy equivalent to their combinatorial analogs, the simplicial complexes $\|\mbox{MacP}(2,n)\|$ and $\|\mbox{MacP}(1,2,n)\|$ respectively. We will prove that $\mbox{Gr}(2, \mathbb{R}^n)$ and $\mbox{Gr}(1,2, \mathbb{R}^n)$ are homeomorphic to $\|\mbox{MacP}(2,n)\|$ and $\|\mbox{MacP}(1,2,n)\|$ respectively}.




\maketitle

\section{Introduction}

An oriented matroid can be thought of as a combinatorial abstraction of a vector space, or of a point configuration or of  real hyperplane arrangements. The theory of oriented matroid comes with analogous notion to linear independence, convexity, general position, and subspaces.

Mn\"{e}v and Ziegler in \cite{zieg:mnev} introduced the poset $G(k,\mathcal{M})$ of rank $k$ strong map images of a rank $n$ oriented matroid $\mathcal{M}$ called the \textit{oriented matroid Grassmannian}. The poset was introduced to serve as a combinatorial model for $\mathrm{Gr}(k, \mathbb{R}^n)$, the space of $k$ dimensional subspaces of $\mathbb{R}^n.$ A special case is when $\mathcal{M}$ is the unique rank $n$ oriented matroid on $n$ elements. The resulting poset $\mathrm{MacP}(k, n)$, is called the \textit{MacPhersonian}. Similarly, the poset of flags $(\mathcal{N}_1, \mathcal{N}_2)$ of oriented matroids, where $\mathcal{N}_1$ is a rank $p$ strong map image of $\mathcal{N}_2$ and $\mathcal{N}_2$ is a rank $k$ strong map image of $\mathcal{M}$ is denoted by $G(p, k, \mathcal{M})$. The poset came up in the work of Babson in \cite{bab:eric} and the work of Gelfand and MacPherson in \cite{Gelf:MacP}.

Mn\"{e}v and Ziegler conjectured that $\|G(k, \mathcal{M})\|$, the geometric realization of the poset $G(k, \mathcal{M})$ has the homotopy type of $\mathrm{Gr}(k, \mathbb{R}^n)$. For $k = 2$, it was proven by Babson in \cite{bab:eric} that $\|G(2, \mathcal{M})\|$ has the same homotopy type as $\mathrm{Gr}(2, \mathbb{R}^n)$. It was also proven in \cite{bab:eric} that $\|\mathrm{G}(1,2, \mathcal{M})\|$ has the same homotopy type as $\mathrm{Gr}(1,2, \mathbb{R}^n)$.

We will show that the complexes $\|\mbox{MacP}(2,n)\|$ and $\|\mbox{MacP}(1,2,n)\|$ are homeomorphic to $\mathrm{Gr}(2,\mathbb{R}^n)$ and $\mathrm{Gr}(1,2,\mathbb{R}^n)$ respectively. It follows from Babson's work in ~\cite{bab:eric}  that the complex $\|\mbox{MacP}(2,n)\|$ and $\mathrm{Gr}(2, \mathbb{R}^n)$ have the same homotopy type, and that the complex $\|\mathrm{MacP}(1,2, n)\|$ has the same homotopy type as $\mathrm{Gr}(1,2,\mathbb{R}^n)$. Also, it can  easily be shown that for $k=1$, $\|\mathrm{MacP}(1,n)\|$ is homeomorphic to $\mathbb{R}P^{n-1}$. 

In Section \ref{sec:back}, we will give basic background on oriented matroids, posets and regular cell complexes that is sufficient for this part of the project.  We will also introduce the maps $\mu: \mathrm{Gr}(k, \mathbb{R}^n) \rightarrow \mathrm{MacP}(k,n)$ mapping a $k$ dimensional subspace to the rank $k$ oriented matroid it determines, and the map $\nu : \mathrm{Gr}(p,k, \mathbb{R}^n) \rightarrow \mathrm{MacP}(p, k, n)$ mapping a flag of subspaces to a flag of oriented matroids. To establish our main assertion about the topology of $\|\mbox{MacP}(2,n)\|$ and $\|\mbox{MacP}(1,2,n)\|$, we will prove in Theorem \ref{main} that the stratification $\{\mu^{-1}(M): M \in \mathrm{MacP}(2,n)\}$ is a regular cell decomposition of $\mathrm{Gr}(2, \mathbb{R}^n)$. Similarly, we will prove in Theorem \ref{main1} that the stratification $\{\nu^{-1}(N, M) : (N , M) \in \mathrm{MacP}(1,2,n)\}$ is a regular cell decomposition for $\mathrm{Gr}(1,2,\mathbb{R}^n)$.

In Proposition \ref{prop1}, we will show that $\mu^{-1}(M)$ is homeomorphic to an open ball. Similarly, $\nu^{-1}(N, M)$ will be shown in Proposition \ref{homeo} to be homeomorphic to an open ball. In Proposition \ref{prop2}, the boundary $\partial \overline{\mu^{-1}(M)}$ of $\mu^{-1}(M)$ will be shown to be $\bigcup_{N< M} \mu^{-1}(N)$ as the union of lower dimensional cells. We have similar result in Proposition \ref{bdry1} for the boundary of $\nu^{-1}(N , M)$. 
	
In Section \ref{sec:shell_1} and Section \ref{sec:shell_2}, we will prove that the intervals $\mathrm{MacP}(2,n)_{\leq M} \cup \{\hat{0}\}$ and $\mathrm{MacP}(1,2,n)_{\leq (N, M)} \cup \{\hat{0}\}$ are isomorphic to the augmented face posets of some shellable regular cell decomposition of spheres. To see this, we will prove that the intervals $\mathrm{MacP}(2,n)_{\leq M} \cup \{\hat{0}\}$ and $\mathrm{MacP}(1,2,n)_{\leq (N, M)} \cup \{\hat{0}\}$ have recursive atom ordering and are thin.

In Section \ref{sec:top_1}  and Section \ref{sec:top_2}, we will prove that the closures $\overline{\mu^{-1}(M)}$ and $\overline{\nu^{-1}(N , M)}$ are topological manifolds whose boundaries are spheres.

In rank $r \geq 3$, and for $M \in \mathrm{MacP}(r,n)$, $\mu^{-1}(M)$ is not necessarily connected. Our argument will also make use of the fact that in rank $2$, $\partial \overline{\mu^{-1}(M)} = \bigcup_{N< M} \mu^{-1}(N)$. This fact called \textit{normality}, is also not necessarily true in rank $r$ for  $r \geq 3$. We will use throughout this part of the project, the \textit{realizability} of rank $2$ oriented matroids; that is every rank $2$ oriented matroid can be obtained from an arrangement of vectors in $\mathbb{R}^2$. This fact is also in general not true for oriented matroids of rank at least $3$. Detailed results on rank $r$ oriented matroids for $r \geq 3$ can be found in  \cite{anders:zieg}.

\section{Oriented Matroids} \label{sec:back}

Suppose $X \in \mathrm{Gr}(p, \mathbb{R}^n)$. We will view elements of $\mathbb{R}^n$ as  $1\times n$ row vectors, so that $X$ is the rowspace of a $p \times n$ matrix.
	
The collection of all sign vectors
$$\{(\mbox{sign}(x_1), \mbox{sign}(x_2), \ldots, \mbox{sign}(x_n)): (x_1, x_2, \ldots, x_n) \in X \}$$ is a collection of sign vectors that are called the \textit{covectors} of an \textit{oriented matroid}. For a covector $C$, the set $\{i \in [n]: C(i) \neq 0\}$ is called the support of $C$. A covector of minimal support is called a \textit{cocircuit}. A formal definition of the covector set of an oriented matroid will be given later in this section.
	
	We can write in terms of column vectors as $X = \mathrm{Rowspace}(v_1 \; v_2 \; v_3 \; \cdots \; v_n)$. Every element of $X$ is of the form $(\alpha\cdot v_1 , \alpha\cdot v_2, \ldots, \alpha\cdot v_n)$ for some $\alpha = (\alpha_1, \alpha_2, \ldots, \alpha_p) \in \mathbb{R}^p$. So, $(\mathrm{sign}(\alpha \cdot v_1), \mathrm{sign}(\alpha \cdot v_2), \ldots \mathrm{sign}(\alpha \cdot v_n))$ is a covector of the oriented matroid $\mathcal{M}$ determined by $X$. 
	
	Let $\{v_{\alpha_i}\}_{\alpha_i}$ be the set of non-zero vectors in $\{v_1, v_2, \ldots v_n\}$. We consider the following arrangement  $(v_{\alpha_i}^{\perp})_{\alpha_i}$ of oriented linear hyperplanes. The arrangement determines a cellular decomposition of $\mathbb{R}^p$. The intersection of the cellular decomposition with $S^{p-1}$ the unit sphere in $\mathbb{R}^p$ gives a cellular decomposition of $S^{p-1}$. A cell in $S^{p-1}$ corresponds to a non-zero covector of  $\mathcal{M}$ and a non-zero covector of $\mathcal{M}$ corresponds to a cell in the cellular decomposition of $S^{p-1}$. The oriented matroid $\mathcal{M}$ with a covector set obtained this way is called a \textit{realizable oriented matroid}. It should be noted that rank $1$ and rank $2$ oriented matroids are realizable, but  oriented matroids of rank at least 3 are not necessarily realizable, details can be found in \cite{anders:zieg}.

Let $X = \mathrm{Rowspace}(v_1\; v_2 \; v_3 \ldots \; v_n)$ as defined earlier and $\mathcal{M}$ the corresponding oriented matroid. We consider the following function $\chi : [n]^p \rightarrow \{+, - , 0\}$  associated to $X$
	$$\chi(i_1, i_2, \ldots , i_p) = \mathrm{sign}(\det(v_{i_1} \; v_{i_2} \cdots v_{i_p}) )$$
The collection $\{\pm \chi\}$ is independent of the choice of basis vectors for $X$. We will write the resulting oriented matroid as $\mathcal{M} = (\pm \chi)$.

	In Figure ~\ref{rank2,3}(b), a rank $3$ oriented matroid is obtained from an essential arrangement of equators in a 2-sphere $S^2$.

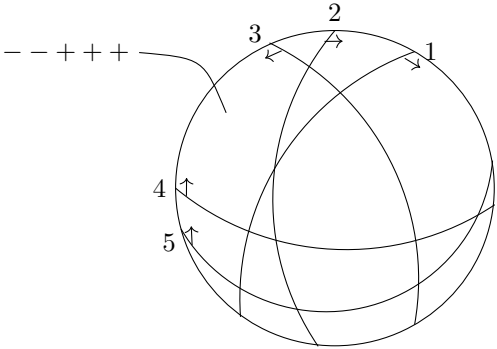
\begin{figure}[htb!]
\begin{subfigure}{0.3\textwidth}
\begin{tikzpicture}
	\begin{scope}[scale=.7]
		\draw (0,0) circle(3);
		\draw (60:3)  arc[radius = 5, start angle= 110, end angle= 184] node[right] at (60:3){$1$};
		\draw[->](62:2.8)--(55:2.8);
		\draw[->](111:2.8)--(118:2.8);
		\draw[->, rotate=32](62:2.8)--(55:2.8);
		\draw[->, rotate=121](62:2.8)--(55:2.8);
		\draw[->, rotate=140](62:2.9)--(55:2.8);
		\draw[rotate=120] (60:3)  arc[radius = 5, start angle= 110, end angle= 184] node[above] at (0:3){$3$};
		\draw[rotate=240] (60:3)  arc[radius = 5, start angle= 110, end angle= 184] node[left] at (300:3){$4$};
		\draw[rotate=30] (60:3)  arc[radius = 5, start angle= 110, end angle= 184] node[above] at (60:3){$2$};
		\draw(195:3) arc[radius=3.15, start angle=210, end angle =355] node[left] at (200:3){$5$};	
		\draw[] (145:4.5).. controls (135:3.5).. (145:2.5) node[left] at (145:4.5) {$--+++$} ;
\end{scope}
\end{tikzpicture}
\end{subfigure}
\caption{An  rank $3$ oriented matroid from an arrangement of equators}
\label{rank2,3}
\end{figure}

Notation: Let $E$ be a finite set and $X,Y\in\{0,+,-\}^E$ be sign vectors. The {\em composition} $X\circ Y$ is defined to be the element of $\{0,+,-\}^E$ with 
\[X\circ Y(e)=\begin{cases} 
X(e)&\mbox{ if $X(e)\neq 0$}\\
Y(e)&\mbox{ otherwise}
\end{cases}\]

\begin{defn} (\cite{anders:zieg}) Let $E$ be a finite set and $\V^* \subseteq\{0,+,-\}^E$. $\V^*$ is the {\em covector set} of an oriented matroid on elements $E$ if it satisfies all of the following.
\begin{enumerate}
\item $\mathbf 0\in\V^*$.
\item If $X\in\V^*$ then $-X\in\V^*$.
\item If $X,Y\in\V^*$ then $X\circ Y\in \V^*$.
\item If $X,Y\in\V^*$, $X\neq -Y$, and $e\in E$ such that $X(e)=-Y(e)\neq 0$, then there 
is a $Z\in\V^*$ such that $Z(e)=0$ and, for each $f\in E$,
\begin{itemize}
\item if $X(f)=Y(f)=0$ then $Z(f)=0$,
\item If $+\in\{X(f),Y(f)\}\subseteq\{0,+\}$ then $Z(f)=+$, and 
\item If $-\in\{X(f),Y(f)\}\subseteq\{0,-\}$ then $Z(f)=-$.
\end{itemize}
\end{enumerate}
If $\V^*$ is the  covector set of an oriented matroid, then the {\em rank} of the oriented matroid is the rank of $\V^*$ as a subposet of $\{0,+,-\}^E$.
\end{defn}

\begin{defn} (\cite{anders:bjo}) Let $E$ be a finite set and $r$ a positive integer. A {\em rank $r$ chirotope on elements $E$} is a nonzero alternating function $\chi:E^r\to\{0,+,-\}$ satisfying the following {\em Grassmann-Pl\"ucker relations}: for each $x_2, x_3,\ldots, x_r, y_0, y_1, \ldots, y_r\in E$, the set
\[\{(-1)^i\chi(y_i, x_2, \ldots, x_r)\chi(y_0, \ldots, \hat y_i,\ldots, y_r)\}\]
 either is $\{0\}$ or contains both $+$ and $-$. 
\end{defn}

\begin{defn}(Basis orientation)(\cite{anders:bjo})
A basis orientation of an oriented matroid $\mathcal{M}$ is a mapping $\chi$ of the set of ordered bases of $\mathcal{M}$ to $\{+1, -1\}$ satisfying the following two properties
\begin{itemize}
    \item $\chi$ is alternating,
    \item For any ordered bases of $\mathcal{M}$ of the form $(e, x_2, x_3, \ldots, x_p)$ and $(f, x_2, x_3, \ldots, x_p)$, $e \neq f$, we have
    $$\chi(e, x_2, x_3, \ldots, x_p) = D(e)D(f),$$
    where $D$ is one of the two opposite cocircuits complementary to the hyperplane spanned by $\{x_2, x_3 , \ldots, x_p\}$ in $\mathcal{M}.$
\end{itemize}
\end{defn}

The following theorem establish the cryptomorphism between the definition of an oriented matroid using covectors and its definition using a chirotope.

\begin{thm} (\cite{jim:law})
Let $p\geq 1$ be an integer and $E$ be a set. A mapping $\chi: E^p \rightarrow \{+1, 0, -1\}$ is a basis orientation of an oriented matroid of rank $p$ on $E$ if and only if it is a chirotope.
\end{thm}

In general, an oriented matroid is obtained from an arrangement of pseudospheres. Figure \ref{pseudo} illustrates an arrangement of pseudospheres.

\begin{thm}(\cite{jim:law}){The Topological Representation Theorem (Folkman-Lawrence 1978)}
	The rank $r$ oriented matroids are exactly the sets $(E,\mathcal{V}^*)$ arising from essential {\em arrangements of pseudospheres} in $S^{r-1}$.
	 \end{thm}

\begin{figure}[htb]
\begin{tikzpicture}
	\begin{scope}[scale=.7]
		\draw (0,0) circle(3);
		\draw[->](64:2.8)--(56:2.8);
		\draw[->](111:2.8)--(118:2.8);
		\draw[->, rotate=32](62:2.8)--(55:2.8);
		\draw[->, rotate=121](62:2.8)--(55:2.8);
		\draw[->, rotate=140](62:2.9)--(55:2.8);
		\draw[rotate=120] (60:3)  arc[radius = 5, start angle= 110, end angle= 184] node[above] at (0:3){$3$};
		\draw[rotate=240] (60:3)  arc[radius = 5, start angle= 110, end angle= 184] node[left] at (300:3){$4$};
		\draw[rotate=30] (60:3)  arc[radius = 5, start angle= 110, end angle= 184] node[above] at (60:3){$2$};
		\draw(195:3) arc[radius=3.15, start angle=210, end angle =355] node[left] at (200:3){$5$};
		
		\draw (60:3)..controls (100:2.5) and (120:1.5)..(120:1)..controls(-60:.9) and (-80:1)..(240:3) node[left] at (60:3.5){$1$};
\draw[] (160:3)..controls (170:2.5) and (200:1.5)..(200:1)..controls(195:.9) and (30:1)..(20:1)..controls(22:1.2) and (-10:2.5)..(-20:3) node[left] at (160:3){$6$};
\draw[->] (162:2.8)--(155:2.8);
\end{scope}
\end{tikzpicture}
\caption{Arrangement of Pseudospheres}
\label{pseudo}
\end{figure}
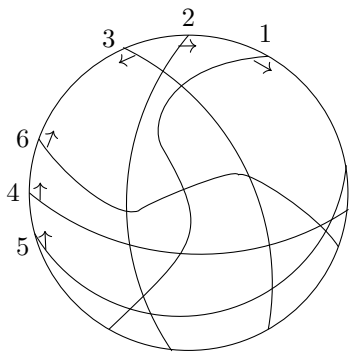

Let $\{+, -, 0\}$ be a poset with the partial order $0 < -$ and $0< +$. The partial order on $\{+, -, 0\}^n$ is component-wise the partial order on $\{+, -, 0\}$.

\begin{defn}(\cite{macp:rob})
    Let $\mathcal{M}$ and $\mathcal{N}$ be two rank $r$ oriented matroids, and $\V^*(\mathcal{M})$ and $\V^*(\mathcal{N})$ the covector sets of $\mathcal{M}$ and $\mathcal{N}$ respectively. We say that $\mathcal{N} \leq \mathcal{M}$ if and only if for every $X \in \V^*(\mathcal{N})$ there exist a $Y \in \V^*(\mathcal{M})$ such that $X \leq Y$. The oriented matroid $\mathcal{M}$ is said to {\em weak map} to $\mathcal{N}$. 
\end{defn}

 \begin{defn}(\cite{macp:rob})
 $\mathrm{MacP}(p,n)$ denotes the poset of all rank $p$ oriented matroids on elements $\{1,2,\ldots, n\}$, with weak map as the partial order. The poset is called the \textit{MacPhersonian} ~\cite{macp:rob}.
 \end{defn}

Let $\mathcal{M}$ be a rank $p$ oriented matroid elements $[n]$ and $\chi: [n]^p \rightarrow \{+ , - ,0\}$ its chirotope. We have the following abstraction of notions from vector spaces and convexity.
\begin{enumerate}[(i)]
    \item \textit{Loop}: An element $i$ is said to be a \textit{loop} of $\mathcal{M}$ if $\chi(i, i_1, i_2, \ldots, i_{p-1}) = 0$ for any $p-1$ tuple $(i_1, i_2, \ldots, i_{p-1}).$

    \item \textit{Basis}: A set $\{i_1, i_2, \ldots, i_p\}$ of size $p$ is said to be a \textit{basis} of $\mathcal{M}$ if and only if $\chi(i_1, i_2, \ldots, i_p) \neq 0$.
    
    \item \textit{Independence}: A set $\{i_1, i_2, \ldots, i_k\}$ is said to be \textit{independent} if it is contained in a basis of $\mathcal{M}$.
    
    \item \textit{Parallel/Anti-parallel}. An non-loop element $i$ is said to be \textit{parallel} to a non-loop element $j$ if for every $p-1$ tuple $(i_1, i_2, \ldots, i_{p-1})$, we have that $\chi(i, i_1, i_2, \ldots, i_{p-1}) = \chi(j, i_1, i_2, \ldots, i_{p-1})$. Similarly, $i$ is said to be \textit{anti-parallel} to $j$ if for every $p-1$ tuple $(i_1, i_2, \ldots, i_{p-1})$, we have that $\chi(i, i_1, i_2, \ldots, i_{p-1}) = -\chi(j, i_1, i_2, \ldots, i_{p-1}).$
    
    \item \textit{Convex Hull}: Let $S$ be a subset of $[n]$. The \textit{convex hull} of $S$ is the set $\{i \in [n]: - \in C(S) \; \mathrm{if}\;  C(i) = - \; \mbox{for all} \; C \in \V^*(\mathcal{M})\setminus \{0\}\}.$ 
\end{enumerate}

An oriented matroid also comes with an abstraction of the notion of subspaces of a vector space. Let $V$ be a rank $k$ subspace of $\mathbb{R}^n$ and $W$ a rank $p$ subspace of $V$. We have that the collection of sign vectors  $\{(\mbox{sign}(x_1), \mbox{sign}(x_2), \ldots, \mbox{sign}(x_n)): (x_1, x_2, \ldots, x_n) \in V\}$ is a subset of the collection $\{(\mbox{sign}(y_1), \mbox{sign}(y_2), \ldots, \mbox{sign}(y_n)): (y_1, y_2, \ldots, y_n) \in W\}$. Let $\mathcal{M}$ be the rank $k$ oriented matroid determined by $V$, and let $\mathcal{N}$ be the rank $p$ oriented matroid determined by $W$. Then $\mathcal{V}^*(\mathcal{N}) \subseteq \mathcal{V}^*(\mathcal{M}) $

\begin{defn}(\cite{zieg:mnev})
    Let $\mathcal{M}$ be a rank $k$ oriented matroid, and $\mathcal{N}$ a rank $p$ oriented matroid. $\mathcal{N}$ is said to be a rank $p$ {\em strong map image} of $\mathcal{M}$ if and only if $\V^*(\mathcal{N}) \subseteq \V^*(\mathcal{M})$.
\end{defn}

\begin{defn}(\cite{zieg:mnev})
Let $\mathcal{M}$ be an oriented matroid. The poset of all rank $p$ oriented matroids that are strong map image of $\mathcal{M}$ is denoted by $\mathrm{G}(p, \mathcal{M})$.
\end{defn}

In Figure \ref{strong_map}, the oriented matroid $\mathcal{N}$ is a rank $2$ strong map image of $\mathcal{M}$. 
\begin{figure}[htb]
\centering
\begin{tikzpicture}
	\begin{scope}[scale=.7]
		\draw (0,0) circle(3);
		\draw (60:3)  arc[radius = 5, start angle= 110, end angle= 184] node[right] at (60:3){$1$};
		\draw[->](68:2.5)--(78:2.5);
		\draw[->](111:2.8)--(118:2.8);
		\draw[->, rotate=32](62:2.8)--(55:2.8);
		\draw[->, rotate=121](62:2.8)--(55:2.8);
		\draw[->, rotate=140](62:2.9)--(55:2.8);
		\draw[rotate=120] (60:3)  arc[radius = 5, start angle= 110, end angle= 184] node[above] at (0:3){$3$};
		\draw[rotate=240] (60:3)  arc[radius = 5, start angle= 110, end angle= 184] node[left] at (300:3){$4$};
		\draw[rotate=30] (60:3)  arc[radius = 5, start angle= 110, end angle= 184] node[above] at (60:3){$2$};
		\draw(195:3) arc[radius=3.15, start angle=210, end angle =355] node[left] at (200:3){$5$};	
\draw (160:3)..controls (170:2.5) and (200:1.5)..(200:1)..controls(195:.9) and (30:1)..(20:1)..controls(22:1.2) and (-10:2.5)..(-20:3) node[left] at (160:3){$6$};
\draw[->] (162:2.8)--(155:2.8);

\draw[dashed] (250:3)..controls (255:2.5) .. (275:0.5).. controls (300:0.1)..(20:2.1).. controls (40:2.5)..(70:3) node[left] at (250:3.4){$\mathcal{N}$};
\end{scope}
\end{tikzpicture}
\caption{strong map image}
\label{strong_map}
\end{figure}
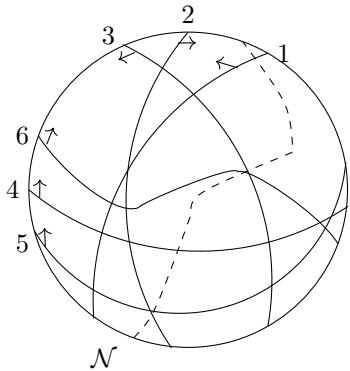

We now define the combinatorial analog of the flag manifold $\mathrm{Gr}(p,k, \mathbb{R}^n)$. The poset came up in the work of Gelfand and MacPherson in \cite{Gelf:MacP} and in the work of Babson in \cite{bab:eric}. 
    
\begin{defn} (\cite{bab:eric})
We define $\mathrm{MacP}(p,k,n)$ as the poset of pairs $(N,M)$ of oriented matroids, where $M$ is a rank $k$ oriented matroid on $n$ elements, and $N$ is a rank $p$ strong map image of $M$. The pair $(N, M)$ is called a \textit{combinatorial flag}.  We say that $(N_1, M_1) \leq (N_2, M_2)$ if and only if $M_1 \leq M_2$ and $N_1 \leq N_2$. 
\end{defn}

As in the case of the map $\mu : \mathrm{Gr}(p,\mathbb{R}^n) \rightarrow \mathrm{MacP}(p,n)$ discussed earlier, we also have the map $\nu: \mathrm{Gr}(k, p, \mathbb{R}^n) \rightarrow \mathrm{MacP}(p,k,n)$ defined by 
$(W, V) \mapsto (N, M)$ where $N$ and $M$ are the oriented matroids determined by subspaces $W$ and $V$ respectively.

\begin{prop}
    Let $M_0\in \mathrm{MacP}(2, n)$ be the rank $2$ oriented matroid whose only basis is $(1, 2)$, and let $N_0$ be a covector of $M_0$. Then $\bigcup_{(N, M) \geq (N_0, M_0)} \nu^{-1}(N, M)$ can be embedded in $\mathrm{Gr}(2, \mathbb{R}^{n+1})$ 
\end{prop}
    
\begin{proof}
Suppose $(N, M) \in \mathrm{MacP}(1,2,n)_{\geq (N_0, M_0)}$. Let $(Y, X) \in \nu^{-1}(N, M)$. Then $X = \mathrm{Rowspace}(e_1 \; e_2 \; v_3\; v_4\; \cdots v_n)$ for some vectors $v_i \in \mathbb{R}^2$. Also, $Y = \mathrm{Rowspace}(\textbf{r})$ for some $\textbf{r} \in X \subseteq \mathbb{R}^n$. Then $\textbf{r}$ is of the form $(\alpha \cdot e_1 \; \alpha \cdot e_2\; \alpha \cdot v_3 \cdots \alpha\cdot v_n )$ for some $\alpha \in \mathbb{R}^2$. There is a unique vector $v_{n+1} \in \mathbb{R}^2$ such that $v_{n+1} \cdot \alpha = 0 $, $\|v_{n+1}\| = 1$ and $0 \leq \mathrm{Arg}(v_{n+1}) < \pi$.

The embedding $\varphi$ is then given by $(Y, X) \mapsto \mathrm{Rowspace}(e_1\; e_2 \; v_3\; v_4 \cdots \; v_{n+1})$.
\end{proof}
Let $(Y, X)$ be any point in $\nu^{-1}(N, M)$. We denote by $\iota(N, M)$ the image of $(Y, X)$ under the map $\mu \circ \varphi$. Let $\chi'$ be a chirotope of $\iota(N, M)$. We observe that $N$ is given by the functions $\pm \chi_{n+1}': [n] \rightarrow \{+, - ,0\}$, where $\chi_{n+1}'(i) = \chi'(n+1, i).$ 
\begin{prop}
Let $(N_0, M_0) \in \mathrm{MacP}(1,2,n)$ such that $\{1,2\}$ is a basis of $M_0$. Then $\mathrm{MacP}_{(N, M) \geq (N_0, M_0)}$ can embedded as a subposet $\mathrm{MacP}(2, n+1)_{\geq \iota(N_0, M_0)}$ of $\mathrm{MacP}(2, n+1)$. 
\end{prop}

\begin{proof}
It follows from the above identifications that $(N_1, M_1) \leq (N_2, M_2)$ if and only if $\iota(N_1, M_1) \leq \iota(N_2, M_2)$. 
\end{proof}

\textbf{SUBP}: Let $(N_0, M_0)$ be in $\mathrm{MacP}(1,2,n)$ and $\iota(N_0, M_0)$ be the corresponding rank $2$ oriented matroid in $\mathrm{MacP}(2,n+1)$ described above. Then the interval $\mathrm{MacP}(1,2,n)_{\geq (N_0, M_0)}$ in $\mathrm{MacP}(1,2,n)$ can be identified with the interval $\mathrm{MacP}(2, n+1)_{\geq \iota(N_0, M_0)}$ in $\mathrm{MacP}(2,n+1)$.
	
We can visualize  an element of the flag manifold $\mathrm{Gr}(1,2,\mathbb{R}^n)$ as in Figure \ref{flag_r} (a). In Figure \ref{flag_r} (b), the sign vectors $\{(+-++-), -(+-++-), 0 \}$ is the covector set of a  rank $1$ oriented matroid determined by the 1 dimensional subspace $Y$.

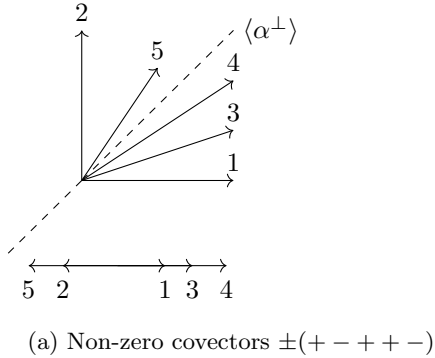
\begin{figure}[!htb]
	\centering
\begin{subfigure}[t]{0.5\textwidth}
	\begin{tikzpicture}
	\path[->] (0,0) edge node[at end, above]{$3$} (2,0.67) ;
	\path[->] (0,0) edge node[at end, above]{$2$} (0,2) ;
	\path[->] (0,0) edge node[at end, above]{$1$} (2,0) ;
	\path[->] (0,0) edge node[at end, above]{$4$} (2,1.33)  ;
	\path[->] (0,0) edge node[at end,  above]{$5$} (1,1.5) ;
	\draw[dashed] (0,0) edge node[at end, right]{$\langle \alpha^{\perp} \rangle$} (2, 2);
	\draw[dashed] (0,0) edge (-1, -1);
	\end{tikzpicture}
\end{subfigure}
~
\begin{subfigure}[t]{0.5\textwidth}
	\begin{tikzpicture}[scale=.9,auto=center,every node/.style={circle}]
	\path[->] (0,0) edge node[at end, below]{$1$} (0.7, 0);
	\path[->] (0,0) edge node[at end, below]{$3$} (1.1, 0);
	\path[->] (0,0) edge node[at end, below]{$4$} (1.6, 0);
	\path[->] (0,0) edge node[at end, below]{$2$} (-0.8, 0);
	\path[->] (0,0) edge node[at end, below]{$5$} (-1.3, 0);
	\end{tikzpicture}
	\caption{Non-zero covectors $\pm (+-++-)$}
\end{subfigure}
\caption{A flag of subspaces, and the rank one oriented matroid  }
\label{flag_r}
\end{figure} 

\subsubsection{Reorientation of an oriented matroid}
We have described a realizable rank $p$ oriented matroid as determined by some arrangement of vectors in $\mathbb{R}^p$. We now discuss the \textit{reorientation} of a rank $p$ oriented matroid.

\begin{defn}
	We say that a rank $p$ oriented matroid $\mathcal{M} = (\pm \chi_\mathcal{M})$ is obtained from a rank $p$ oriented matroid  $ \mathcal{N} = (\pm \chi_\mathcal{N})$  by a \textit{reorientation} of an element $i$ if $\chi_\mathcal{M}(i_1, i_2, \ldots, i_p) = \chi_\mathcal{N}(i_1, i_2, \ldots, i_p)$ for all $p$-tuples $(i_1, i_2, \ldots, i_p)$ such that $i \notin \{i_1, i_2, \ldots, i_p\}$, and $\chi_\mathcal{M}(i,j_1, j_2,\ldots, j_{p-1}) = - \chi_\mathcal{N}(i,j_1, j_2, \ldots, j_{p-1})$ for all $p-1$ tuples $(j_1, j_2, \ldots j_{p-1}).$ 
\end{defn}

In the language of vector arrangements, reorientation of an element $i$ is simply that if $(w_1, w_2, w_3, \ldots,w_i, \ldots , w_n)$ is a vector arrangement for $\mathcal{N}$, then $(w_1, w_2, w_3, \\ \ldots , -w_i, \ldots , w_n)$ is a vector arrangement for $\mathcal{M}$.

It then follows that if $\mathcal{M}$ is obtained from $\mathcal{N}$ by a reorientation or a sequence of reorientations, then $\mu^{-1}(\mathcal{M})$ is homeomorphic to $\mu^{-1}(\mathcal{N})$. Another useful observation is the following lemma: 
\begin{lemma}
	Suppose $\mathcal{N}$ is obtained from $\mathcal{M}$ by reorientations of some elements in $[n]$. Then the posets $(\hat{0}, \mathcal{M})$ and $(\hat{0}, \mathcal{N})$ are isomorphic.
\end{lemma}
\begin{proof}
	Suppose $\mathcal{N}$ is obtained from $\mathcal{M}$ by reorienting elements in $A \subseteq [n]$. Then the required poset  isomorphism $R_A: (\hat{0}, \mathcal{M}) \rightarrow (\hat{0}, \mathcal{N})$ is obtained by taking $R_A(Y)$ as a reorientation of the rank $2$ oriented matroid $Y$ by elements in $A$.
\end{proof}

It should also be noted that the homeomorphism type of $\mu^{-1}(\mathcal{M})$ is unchanged by relabelling the elements of $\mathcal{M}$. 

\section{Posets, Regular Cell Complexes and Topological Balls }

\subsection{Posets and Recursive atom ordering}
Associated to every poset $P$ is a simplicial complex $\|P\|$ whose $n$-simplices are chains $x_0 < x_1 < \cdots < x_n$ for some $x_i$ in $P$. $\|P\|$ is called the \textit{order complex} of $P$. We will be studying the topology of the order complex of the posets $\mathrm{MacP}(2,n)$ and $\mathrm{MacP}(1,2,n)$.

\begin{defn}(\cite{anders:wachs})
A finite poset $P$ is said to be \textit{semimodular} if it is bounded, and whenever two distinct elements $u, v$ both cover $x \in P$, there is a $z \in P$ that covers both $u$ and $v$. The poset $P$ is defined to be \textit{totally semimodular} if it is bounded and every interval in $P$ is semimodular.  
\end{defn}

\begin{defn}(\cite{anders:wachs})
Let $P$ be a poset. $P$ is said to be \textit{thin} if every interval of length $2$ in $P$ has exactly four elements.
\end{defn}

\begin{defn}(\cite{anders:wachs})\label{raon}
	A graded poset $P$ is said to admit a \textit{recursive atom ordering} if the length of $P$ is 1 or if the length of $P$ is greater than 1 and there is an ordering $a_1, a_2, \ldots , a_t$ of the atoms of $P$ which satisfies:
	
	\begin{enumerate}[(i)]
		\item For all $j = 1, 2, \ldots , t \; [a_j, \hat{1}]$ admits a recursive atom ordering in which the atoms of $[a_j, \hat{1}]$ that comes first in the ordering are those that cover some $a_i$ where $i < j$.
		
		\item For all $i < j$, if $a_i, a_j < y$, then there is a $k < j$ and an element $z \leq y$ such that $z$ covers $a_k$ and $a_j$.
	\end{enumerate} 
\end{defn}

The following theorem and proof appears in the work of Bjorner and Wachs (\cite{anders:wachs}). We will give below their proof of the only if direction.

\begin{thm}\label{bw}(\cite{anders:wachs})
 A graded poset $P$ is totally semimodular if and only if for every interval $[x,y]$ of $P$, every atom ordering of $[x,y]$ is a recursive atom ordering.
\end{thm}
\begin{proof}{Proof of the only if }
	Let $P$ be a totally semimodular poset with length greater than $1$. If $[x,y] \neq P$, then $[x,y]$ is totally semimodular,  and by induction every atom ordering of $[x,y]$ is a recursive atom ordering.
	
	Let $a_1, a_2, \ldots, a_n$ be any atom ordering in $P$, since every atom ordering in $[a_j, \hat{1}]$ is recursive, order the atoms of $[a_j, \hat{1}]$ so that those that cover $a_i$ for some $i < j$ come first. Let $y \geq a_i, a_j$. Since $P$ is totally semimodular, there is a $z \in P$ that covers both $a_i$ and $a_j$.   
\end{proof}

\subsection{Regular cell complexes}

\begin{defn}
	A \textit{regular CW complex} $X$ is a collection of disjoint open cells $\{e_{\alpha}\}$ whose union is $X$ such that :
	\begin{enumerate}[(i)]
		\item $X$ is Haursdorff
		
		\item For each open $m$-cell $e_\alpha$, there exists a continuous map $f_\alpha : D^m \rightarrow X$ that is a homeomorphism onto $\overline{e_\alpha}$ , maps the interior of $D^m$ onto $e_\alpha$, and maps the boundary $\partial D^m$ into a finite union of open cells, each of dimension less than $m$.
	\end{enumerate} 
\end{defn}

Let $\{e_{\alpha}\}$ be a regular cell decomposition of $X$ and $F(X)$ the face poset of the decomposition, with $e_\alpha \leq e_\tau$ if and only if $e_\alpha \subseteq \partial \overline{e_\tau}$. The complex $\|F(X)\|$ is homeomorphic to $X$.

\begin{thm}\label{bjorner}(\cite{anders:wachs})
	If a poset $P$ containing $\hat{0}$ and $\hat{1}$ is thin and admits a recursive atom ordering, then $P$ is the augmented face poset of a regular cell decomposition of a $PL$ sphere. 
\end{thm}

\subsubsection{Cell Collapse}

Hersh in \cite{hersh:pat} introduced a general class of collapsing maps which may be performed sequentially on a polytope and preserving homeomorphism type. Each such map is defined by first covering a polytope face with a family of parallel lines or generally a family of parallel-like segments across which the face is collapsed. An example \cite{hersh:pat} of such a collapsing map is given below:

\begin{example} (\cite{hersh:pat})
Let $\Delta_2$ be the convex hull of $(0,0), (1,0), (0, 1/2)$ in $\mathbb{R}^2$, and let $\Delta_1$ be the convex hull of $(0,0)$ and $(1,0)$ in $\mathbb{R}^2$. There is a continous and surjective function $g: \mathbb{R}^2 \rightarrow \mathbb{R}^2$ that acts homeomorphically on $\mathbb{R}^2\setminus \Delta_2$ sending it onto $\mathbb{R}^2 \setminus \Delta_1$. The simplex $\Delta_2$ is covered by vertical line segments with an end point in $\Delta_1$. The map $g$ has the property that it  maps each vertical line segment to its end point in $\Delta_1$.

Let $R = \{(x,y): -1 \leq x \leq 1, 0 \leq y \leq 1\}$. For $0 \leq x \leq 1$ and $0 \leq y \leq -x/2 + 1/2$, let $g(x,y) = (x,0)$. For $0 \leq x \leq 1$ and $-x/2 + 1/2 \leq y \leq 1$, let $g(x,y) = (x, \frac{y- 1/2+x/2}{1/2+x/2})$. For $-1 \leq x \leq 0$ and $0 \leq y \leq -x/2 + 1/2$, let $h(x,y) = (x, y\frac{-x}{-x/2+1/2})$. For $-1 \leq x \leq 0$ and $-x/2 + 1/2 \leq y \leq 1$, let $g(x,y) = (x, -1+2y)$. Let $g$ acts as the identity outside $R$.

\end{example}

Hersh gave a formal definition of a collapsing map below
\begin{defn}(\cite{hersh:pat})
Let $g: X \rightarrow Y$ be a continuous, surjective function with the quotient topology on $Y$. That is, open sets in $Y$ are the sets whose inverse images under $g$ are open in $X$. We call such a map $g$ an \textit{identification map}.
\end{defn}

\begin{defn} (\cite{hersh:pat})
Given a finite regular CW complex $K$ on a set $X$ an an open cell $L$ in $K$, define a \textit{face collapse or cell collapse} of $\overline{L}$ onto $\overline{\tau}$ for $\tau$ an open cell contained in $\partial L$ to be an identification map $g : X \rightarrow X$ such that:
\begin{enumerate}[(i)]
    \item Each open cell of $\overline{L}$ is mapped surjectively onto an open cell of $\overline{\tau}$ with $L$ mapped onto $\tau$.
    \item $g$ restricts to a homeomorphism from $K \setminus \overline{L}$ to $K \setminus \overline{\tau}$ and acts homeomorphically on $\overline{\tau}$.
    \item The images under $g$ of the cells of $K$ form a regular CW complex with new characteristic maps obtained by composing the original characterisitc maps of $K$ with $g^{-1}: X \rightarrow X$ for those cells of $K$ contained either in $\overline{\tau}$ or in $K \setminus \overline{L}$.
\end{enumerate}
\end{defn}

\begin{defn}(\cite{hersh:pat})
Let $K_0$ be a convex polytope, and let $\mathcal{C}_i^0$ be a family of parallel line segments covering a closed face $L_i^0$ in $\partial K_0$ with the elements of $\mathcal{C}_i^0$ given by linear functions $c: [0,1] \rightarrow L_i^0$. Suppose that there is a pair of closed faces $G_1, G_2$ in $\partial L_i^0$ with $c(0) \in G_1$ and $c(1) \in G_2$ for each $c \in \mathcal{C}_i^0$ and there is a composition $g_i \circ \cdots \circ  g_1$ of face collapses on $K_0$ such that:
\begin{enumerate}[(i)]
    \item $g_i \circ \cdots \circ  g_1$ acts homeomorphically on $\mathrm{int}(L_i^0)$
    \item For each $c \in \mathcal{C}_i^0$, $g_i \circ \cdots \circ g_1$ either sends $c$ to a single point or acts homeomorphically on $c$.
    \item Suppose $g_i \circ \cdots \circ g_1(c(t)) = g_i \circ \cdots \circ g_1(c'(t'))$ for $c\neq c' \in \mathcal{C}_i^0$ and some $(t,t')  \neq (1,1)$. Then $t = t'$, and for each $t \in [0,1]$ we have $g_i \circ \cdots \circ g_1(c(t)) = g_i \circ \cdots \circ g_1(c'(t'))$.
\end{enumerate}
Then call $\mathcal{C}_i = \{g_i \circ \cdots \circ g_1(c(t))| c \in \mathcal{C}_i^0\}$ a family of parallel-like curves on the closed cell $L_i = g_i \circ \cdots \circ g_1(L_i^0)$ of the finite regular CW complex $K_i = g_i \circ \cdots \circ g_1(K_0)$.
\end{defn}

\begin{thm}(\cite{hersh:pat})\label{hersh:thm}
Let $K_0$ be a convex polytope. Let $g_1, \ldots, g_i$ be collapsing maps with $g_j: X_{K_{j-1}} \rightarrow X_{K_j}$ for regular CW complexes $K_0, K_1, \ldots, K_i$ all having the underlying space $X$. Suppose that there is an open cell $L_i^0$ in $\partial K_0$ upon which $g_i \circ \cdots \circ g_1$ acts homeomorphically and a collection $\mathcal{C} = \{g_i \circ \cdots \circ g_1(c)| c \in \mathcal{C}_i^0\}$ of parallel-like curves covering $\overline{L_i}$ for $L_i = g_i \circ \cdots \circ g_1(L_i^0) \in K_i$. Then there is an identification map $g_{i+1}: X_{K_i} \rightarrow X_{K_{i+1}}$ specified by $\mathcal{C}$.  
\end{thm}

The sequence of cell collapses needed in the proof of Theorem \ref{top} will be of the form illustrated by the example below.
\begin{example}\label{ex}
	Let $K_0 = \Delta^3  \times [0,1] \times [1,2]\times [0,1]$, where $\Delta^3 = \{(x_1, x_2, x_3): 0 \leq x_1 \leq x_2 \leq x_3 \leq \frac{\pi}{2} \}$ is a $3$-simplex. Let $L_0 = \Delta^3 \times \{0\} \times [1,2] \times [0,1]$ be a face of $K_0$. Let $G_1^0 = \{(0, x_2, x_3): 0= x_1 \leq x_2 \leq x_3 \leq \frac{\pi}{2}\} \times \{0\} \times [1,2] \times [0,1]$ and $G_2^0 = \{(x_1, x_2, x_3): 0 \leq x_1 = x_2 \leq x_3 \leq \frac{\pi}{2}\} \times \{0\} \times [1,2] \times [0,1]$. Then $G_1^0$ and $G_2^0$ are faces of $L_0$. 
	
	Let $(x,r) = ((0, x_2, x_3), 0, r_2, r_3) \in G_1^0$, and $c_{(x,r)}$ is a curve defined by:
	$$c_{(x,r)}: [0,1] \rightarrow L_0 : t \mapsto ((t\cdot x_2, x_2, x_3), (0, r_2, r_3)).$$
	\end{example}
The end point of the curve $c_{(x,r)}$ lies in $G_2^0$. Let $\mathcal{C}_0 = \{c_{(x,r)}: (x,r) \in G_1^0\}$. Then $\mathcal{C}_0$ is a collection of parallel-like line segments covering the face $L_0$. By Theorem \ref{hersh:thm}, there is an identification map $g_1: K_0 \rightarrow K_0$ specified by $\mathcal{C}_0$. The map has the property that the curves in $\mathcal{C}_0$ are mapped to their end points in $G_2^0$. 

Let $\sim_1$ be a relation on $K$ defined by $(x,r) \sim_1 (x', r')$ if and only if $g_1(x, r) = g_1(x', r')$. Then the quotient $K_1 = K/\sim_1$ is homeomorphic to $K$.  

Also, let $L_1 = \Delta^3 \times [0,1] \times [1,2] \times \{0\}$ be a face of $K_0$, $G_1^1 = \{(x_1, x_2, x_3): x_1 \leq x_2 \leq x_3 \leq \frac{\pi}{2}\} \times [0,1] \times [1,2] \times \{0\}$, and $G_2^1 = \{(x_1, x_2, \frac{\pi}{2}): x_1 \leq x_2 \leq x_3 = \frac{\pi}{2}\} \times [0,1] \times [1,2] \times \{0\}$. Then $G_1^1$ and $G_2^1$ are faces of $L_1$. Let $\mathcal{C}^0_1$ be a collection of parallel line segments covering $L_1$ and with endpoints in $G_1^1$ and $G_2^1$. Then $\mathcal{C}_1 = \{g_1(c)| c \in \mathcal{C}_1^0\}$ is a collection of parallel-like segments covering $g_1(L_1)$.  By Theorem \ref{hersh:thm}, there is an identification map $g_2: K_1 \rightarrow K_1$ specified by $\mathcal{C}_1$. Let $\sim_2$ be a relation on $K_1$ defined by $[(x,r)] \sim_2 [(x', r')]$ if and only if $g_2([(x,r)]) = g_2([(x', r')])$. Hence, $K_2 = K_1/\sim_2$ is homeomorphic to $K_1$.

In Theorem \ref{top}, we will apply homeomorphism-preserving cell collapses on the boundary of a closed ball of the form $$B_M = \{(\theta_1, \theta_2, \ldots, \theta_{n_1}) : 0 \leq \theta_1 \leq \theta_2 \cdots \leq \theta_{n_1} \leq\frac{\pi}{2}\} 
	\times $$ $$\{(\eta_1, \eta_2, \ldots, \eta_{n_2}) : \frac{\pi}{2} \leq \eta_1 \leq \eta_2 \leq \cdots \leq \eta_{n_2}\leq\pi \} \times [a ,b]^{L_M-2}.$$The identification $\sim$ as in Example \ref{ex} is given here as in Figure \ref{viz_0}.

\begin{figure}[!htb]
	 	\centering
	 \begin{subfigure}[t]{0.5\textwidth}
	 	\begin{tikzpicture}
	 	\path[->] (0,0) edge node[at end, right = 0.5mm]{$3$} (3/1.5,1/1.5) ;
	 	\path[->] (0,0) edge node[at end, above]{$7$} (0,2.5) ;
	 	\path[->] (0,0) edge node[at end, right= 0.5mm]{$1$} (2.5,0) ;
	 	\path[->] (0,0) edge node[at end, above]{$4$} (3/1.5,2/1.5)  ;
	 	\path[->] (0,0) edge node[at end, above]{$6$} (1.1, 1.7) ;
	 	\path[dashed][->] (0,0) edge node[at end, below=1mm, right]{$2$} (2.0, 0.50) ;
	 	\path[dashed][->] (0,0) edge node[at end, above]{$5$} (1.45, 1.6);
	 	\path[->] (0,0) edge node[at end, above,  right = 15mm]{$\sim$} (3/1.5, 2/1.5);
	 	\end{tikzpicture}
	 	\caption{$\theta_5 = a, r_5 = 0, r_2 = 0$}
	\end{subfigure}
	~
	    \begin{subfigure}[t]{0.5\textwidth}
	 	\begin{tikzpicture}
	 	\path[->] (0,0) edge node[at end, right = 0.5mm]{$3$} (3/1.5,1/1.5) ;
	 	\path[->] (0,0) edge node[at end, above]{$7$} (0,2.5) ;
	 	\path[->] (0,0) edge node[at end, below]{$1$} (2.5,0) ;
	 	\path[->] (0,0) edge node[at end, above, right = 0.5mm]{$4$} (3/1.5,2/1.5)  ;
	 	\path[->] (0,0) edge node[at end,above]{$6$} (1.1, 1.7) ;
	 	\path[dashed][->] (0,0) edge node[at end, above = 0.5mm ]{$5$} (1.8, 1.4);
	 	\path[->] (0,0) edge node[at end, above,  right = 10mm]{$\equiv$} (3/1.5, 2/1.5);
	 	\path[dashed][->] (0,0) edge node[at end, right]{$2$} (2.4, 0.2);
	 	\end{tikzpicture}
	 	\caption{$\theta_5 = b, r_5 = 0, r_2 = 0$}
	\end{subfigure}
	~
 	 \begin{subfigure}[t]{0.5\textwidth}
 		\begin{tikzpicture}
 		\path[->] (0,0) edge node[at end, right = 0.5mm]{$3$} (3/1.5,1/1.5) ;
 		\path[->] (0,0) edge node[at end, above = 0.07mm]{$7$} (0,2.5) ;
 		\path[->] (0,0) edge node[at end, right= 0.5mm]{$1$} (2.5,0) ;
 		\path[->] (0,0) edge node[at end, above]{$4$} (3/1.5,2/1.5)  ;
 		\path[->] (0,0) edge node[at end,above]{$6$} (1.1, 1.7) ; 
 		\end{tikzpicture}
 		\caption{$N$}
 	\end{subfigure}
	 	\caption{An identification on the boundary of $\mu^{-1}(M).$}
	 	\label{viz_0}
	 \end{figure}
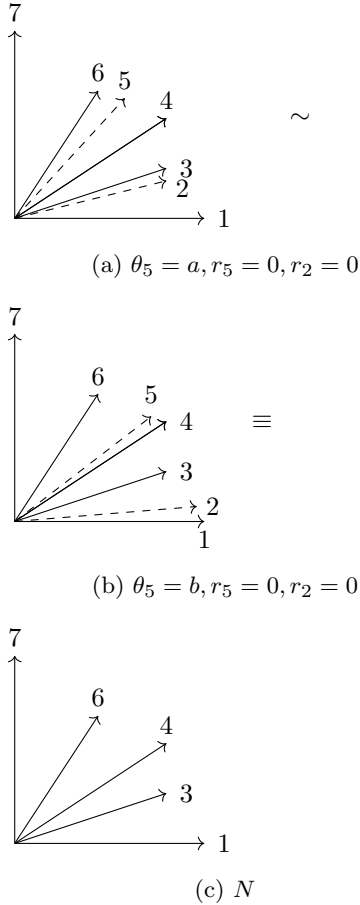
The identification as in Figure \ref{viz_0} on the boundary $r_2 = 0, r_5 = 0$ gives the collapse of the 5-simplex $\{(\theta_2, \theta_3 , \theta_4, \theta_5, \theta_6) : 0 \leq \theta_2 \leq \theta_3 \leq \cdots \theta_6 \leq \frac{\pi}{2}\}$ onto the 3-simplex $\{(\theta_2, \theta_3, \theta_4, \theta_5 , \theta_6) : 0 = \theta_2 \leq \theta_3 \leq \theta_4 = \theta_5 \leq \theta_6 \leq \frac{\pi}{2}\}$.

\subsubsection{Topological Balls}

\begin{thm}(\cite{can:jc})\label{cantrel}
   	Let $S$ be an $(n-1)$ sphere in $S^n$ and $H$ a component of $S^n - S$. If $\overline{H}$ is a topological manifold with boundary, then $\overline{H}$ is homeomorphic to an $n$-ball.
   \end{thm}

   \begin{cor}\label{corb} Suppose $\mathrm{int}(D^n)$ is the interior of a closed unit ball centered at the origin. If $\overline{H} \subset \mbox{int}(D^n)$ is a topological manifold whose boundary $S$ is an $(n-1)$ sphere, then $\overline{H}$ is homeomorphic to a closed unit ball $D^n$.
   \end{cor}
   
   \begin{defn}
   	Let $M$ be a topological manifold  and $S \subseteq M$. The subset $S$ is said to be \textit{collared} in $M$ if there is a neighborhood $N(S) \subseteq M$ and a homeomorphism $h : S \times I \rightarrow N(S)$ satisfying $h(x,1) = x$.
   \end{defn}
\begin{thm}(\cite{brown:mort}) \label{morton}
	The boundary of a topological manifold $M$ with boundary is collared in $M$.
\end{thm}
 
\begin{prop}\label{ball}
Suppose $G \subset X$ is homeomorphic to an open ball and $\overline{G}$ is a topological manifold whose boundary is a sphere. Then $\overline{G}$ is homeomorphic to a closed ball. 
\end{prop}
   
   \begin{proof}
   	Let $h_1 :  G \rightarrow \mbox{int}(D^n)$  be an homeomorphism to the open unit ball and $S = \partial \overline{G}$ the boundary of $G$. As $\overline{G}$ is a topological manifold whose boundary is $S$, by Theorem \ref{morton}, there is a collared neighborhood $N(S)$ of $S$ in $\overline{G}$ and a homeomorphism $h_2 : S \times I \rightarrow N(S) \subseteq \overline{G}$ with $h_2(S\times \{1\}) = S$. Let $A = h_2(S\times \{0\} )  \subset G$, and $\overline{A}$ the subset of $G$ whose boundary is $A$. 
   	
   	So $A' = h_1(\overline{A}) \subset \mbox{int}(D^n)$ is a topological manifold whose boundary is the sphere $h_1(A)$. By Corollary \ref{corb}, $A'$ is homeomorphic to a closed ball. 
   	
   	We will denote by $h_3 : A'  \rightarrow D_{\frac{1}{2}}$ the homeomorphism of $A'$ to a closed ball, where $D_{\frac{1}{2}}$ is a closed ball of radius $\frac{1}{2}$ centered at the origin. Define: $$h_4 : N(S) \equiv S \times I \rightarrow D^n : \;   (x,t) \mapsto (1 + t)\cdot h_3\circ h_1 \circ h_2(x,0) .$$ 
   	
   	$h_4$  gives a homeomorphism from $N(S) \subseteq \overline{G}$ to the annulus in $D^n$ the unit ball. The map  $h: \overline{G} \rightarrow D^n$ given by $ h = (h_3\circ h_1) \cup h_4$ is a homeomorphism by the Gluing Lemma. \cite{jim:law}  
   \end{proof}

\section{The stratification $\mu^{-1}(M)$}\label{sec:ball_1}

\begin{thm}\label{main} $\{\mu^{-1}(M) : M \in \; \mbox{MacP}(2, n)\}$ is a regular cell decomposition of $Gr(2,\mathbb{R}^n)$.
	\end{thm}

Let $M$  be a rank $2$ oriented matroid. We will first determine the topology  of $\mu^{-1}(M) = \{X\in Gr(2, \mathbb{R}^n): (\pm \chi_X) = M\}$. We know from the background on oriented matroids in Section \ref{sec:back} that the topology of $\mu^{-1}(M)$ is invariant under relabeling and reorientation of elements of the oriented matroid. We may thus assume that $\{1, 2\}$ is a basis of $M$ for the rest of the proofs in this chapter. We may also assume that for any $X \in \mu^{-1}(M)$, $X$ can be uniquely expressed as $X = \mbox{Rowspace}(e_1 \; e_2 \; v_3 \; v_2 \; \cdots v_n)$ satisfying $0 \leq \mathrm{Arg}(v_k) < \pi$ for a non-zero vector $v_k$. For such a non-zero vector $v_k$, we let $\theta_k = \mathrm{Arg}(v_k)$ and $r_k = |v_k|$, so that $v_k = r_ke^{i\theta_k}$.

Arranging the arguments of the non-zero vectors in an increasing order as $0 < \theta_{l_1} < \theta_{l_2} < \cdots < \theta_{l_{n_1}} < \frac{\pi}{2} $, $\frac{\pi}{2} < \theta_{j_1} < \theta_{j_2} < \cdots < \theta_{j_{n_2}} < \pi $, we can thus uniquely represent $X$ as  $( (\theta_{l_1}, \theta_{l_2}, \ldots, \theta_{l_{n_1}}), (\theta_{j_1}, \theta_{j_2}, \ldots, \theta_{j_{n_2}}), (r_\lambda)_{\lambda} )$. An example of such an identification is given in Figure \ref{real}.
\begin{figure}[htb!]
\begin{tikzpicture}
	\path[->] (0,0) edge node[at end, above]{$3$} (2,0.67) ;
	\path[->] (0,0) edge node[at end, above]{$2$} (0,2) ;
	\path[->] (0,0) edge node[at end, above]{$1$} (1.5,0) ;
	\path[->] (0,0) edge node[at end, above]{$9$} (2.5,0) ;
	\path[->] (0,0) edge node[at end, above]{$4$} (2,1.33)  ;
	\path[->] (0,0) edge node[at end,  above]{$5$} (1,1.35) ;
	\draw[dashed] (0,0) edge (-2,0); 
	\path[->] (0,0) edge node[at end, above]{$6$}(-1, 1);
	\path[->] (0,0) edge node[at end, above]{$7$}(-1.5, 1.5);
	\path[->] (0,0) edge node[at end, above]{$8$}(-1.5, 0.5);
	\end{tikzpicture}
	\caption{$X \; \mbox{as a point}\; ((\theta_3, \theta_4, \theta_5), (\theta_6, \theta_8), (r_i)_{1\leq i \leq 9})$}
	\label{real}
\end{figure}
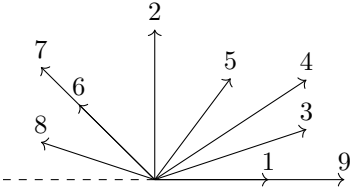
The above representation of an element $X \in \mu^{-1}(M)$ gives an identification

$(\textbf{ID1}) \; \; \mu^{-1}(M) \equiv A_M = \mathrm{Interior}(\{(\theta_{l_1}, \theta_{l_2}, \ldots, \theta_{l_{n_1}}) : 0 \leq \theta_{l_1} \leq \theta_{l_2} \cdots \leq \theta_{l_{n_1}} \leq\frac{\pi}{2}\} \times \{(\theta_{j_1}, \theta_{j_2}, \ldots, \theta_{j_{n_2}}) : \frac{\pi}{2} \leq \theta_{j_1} \leq \theta_{j_2} \leq \cdots \leq \theta_{j_{n_2}}\leq\pi \} \times [0, \infty)^{L_M-2})$

\begin{nota}
Let $M \in \mathrm{MacP}(2,n)$ and $i$ a non-loop in $M$. $l_M$ denotes the number of non-loops element in $M$. $p_M$ denotes the number of distinct parallel/anti-parallel classes of $M$. $P_M(i)$ denotes the set of elements that are parallel/anti-parallel to $i$.
\end{nota}

The following proposition then follows from the above identification of $\mu^{-1}(M).$

\begin{prop}\label{prop1}
	Let $M \in \mbox{MacP}(2,n)$. Then $\mu^{-1}(M)$ is homeomorphic to an  open ball of dimension $h(M) = l_M + p_M -4$. 
\end{prop}

\begin{prop}\label{cor1}
	 Let $(\pm\chi_N, N), (\pm\chi_{M}, M)$ be two rank $2$ oriented matroids. Suppose that $M$ covers $N$. Then exactly one of the following possibilities holds:
	\begin{enumerate}[]
		\item[] \textbf{CR1} There is exactly one $i$ such that $|P_M(i)|\geq 2$ in  $M$, $i$ is a loop in $N$ and $\chi_N(j,k) = \chi_M(j,k)$ for $j,k \neq i$.
		
		\item[] \textbf{CR2} There are exactly two distinct parallel/anti-parallel classes $P_M(i)$ and $P_M(j)$ in $M$ such that the following holds: $$\chi_N(a,b) = \left\{ \begin{array}{ccc}
		0 & \mbox{if} & a, b \in P_M(i) \cup P_M(j)\\
		\chi_M(a,b) & & \mbox{otherwise}
		\end{array}\right\}$$
	\end{enumerate}

\end{prop}

\begin{proof}
We first check that if $M$ and $N$ satisfy \textbf{CR1} or \textbf{CR2} then $M$ covers $N$. Suppose $M$ and $N$ satisfy \textbf{CR1}; that is $\chi_N(a,b) = \chi_M(a,b)$ for all $a, b \neq i$ and $i$ is a loop of $N$. Let $N'$ be a rank $2$ oriented matroid such that $N < N' \leq M$. Then $i$ is not a loop in $N'$ as otherwise $N' \leq N$. We thus have that $\chi_{N'}(i, k) \neq 0$ if and only if $\chi_M(i,k) \neq 0$. Hence $N' = M$. 

Also, suppose that $M$ and $N$ satisfy \textbf{CR2}. Then there are adjacent parallel classes $P_M(i)$, $P_M(j)$ with $\chi_N$ satsfying the definition in \textbf{CR2}. Suppose there is an $N'$ satsfying $N < N' \leq M$. Then $\chi_{N'}(i,j) \neq 0$ as otherwise $N' \leq N$. So, $\chi_{N'}(a,b) = \chi_M(a,b)$ for all $a \in P_M(i)$ and $b \in P_M(j)$. Hence $N' = M.$

Let $N_0 < M$, and let $i,j$ be non-loops in $N_0$. $\chi_M(i, j) > 0 $ implies that $\chi_{N_0}(i,j) \geq 0$. That is, for any subspace realization $\mathrm{Rowspace}(e_1\; e_2 \; v_3 \; v_4 \cdots \; v_n ) \in \mu^{-1}(M)$, $\mathrm{Arg}(v_i) < \mathrm{Arg}(v_j)$ implies that $\mathrm{Arg}(w_i) \leq \mathrm{Arg}(w_j)$ for any vector realization $\mathrm{Rowspace}(e_1\; e_2 \; w_3\; w_4 \cdots \; w_n) \in \mu^{-1}(N_0)$. Also $\chi_M(i,j) = 0$ implies that $\chi_{N_0}(i, j) = 0$. In otherwords, $\mathrm{Arg}(v_i) = \mathrm{Arg}(v_j)$ implies that $\mathrm{Arg}(w_i) = \mathrm{Arg}(w_j)$ if $i, j$ are non-loops in $N_0$. 
\begin{figure}[!htb]
\begin{subfigure}[t]{0.3\textwidth}
\begin{tikzpicture}
\path[->] (0,0) edge node[at end, below, right= 1mm]{$3$} (1.7,0.5) ;
\path[->] (0,0) edge node[at end, above]{$2$} (0,1.5) ;
\path[->] (0,0) edge node[at end, below]{$1$} (1.5,0) ;
\path[->] (0,0) edge node[at end, above]{$4$} (1.33,0.89)  ;
\path[->] (0,0) edge node[at end,  above]{$5$} (1,1.35) ;
\path[->, dashed] (0,0) edge node[at end, above]{$l_j$} (0.5, 1.5) ;
\path[->, dashed] (0,0) edge node[at end, right= 0.1mm]{$l_i$} (1.5, 0.7) ;
\end{tikzpicture}
\end{subfigure}
~
\begin{subfigure}[t]{0.3\textwidth}
\begin{tikzpicture}
\path[->] (0,0) edge node[at end, above]{$3$} (1.7,0.67) ;
\path[->] (0,0) edge node[at end, above]{$2$} (0,1.5) ;
\path[->] (0,0) edge node[at end, above]{$1$} (1.5,0) ;
\path[->] (0,0) edge node[at end, above]{$4$} (1.33,0.89)  ;
\path[->] (0,0) edge node[at end,  above]{$5$} (1,1.35) ;
\path[->, dashed] (0,0) edge node[at end, above]{$l_j$} (1.3, 1.755) ;
\path[->, dashed] (0,0) edge node[at end, right= 0.1mm]{$l_i$} (2.04, 0.804) ;
\end{tikzpicture}
\caption{$N'$}
\end{subfigure}
~
\begin{subfigure}[t]{0.3\textwidth}
\begin{tikzpicture}
\path[->] (0,0) edge node[at end, above]{$3$} (1.7,0.67) ;
\path[->] (0,0) edge node[at end, above]{$2$} (0,1.5) ;
\path[->] (0,0) edge node[at end, above]{$1$} (1.5,0) ;
\path[->] (0,0) edge node[at end, below= 0.5mm]{$4$} (0.77, 1.04)  ;
\path[->] (0,0) edge node[at end,  above]{$5$} (1,1.35) ;
\path[->, dashed] (0,0) edge node[at end, above]{$l_j$} (1.3, 1.755) ;
\path[->, dashed] (0,0) edge node[at end, right= 0.1mm]{$l_i$} (2.04, 0.804) ;
\end{tikzpicture}
\caption{$N_0$}
\end{subfigure}
~
\begin{subfigure}[t]{0.3\textwidth}
\begin{tikzpicture}
\path[->] (0,0) edge node[at end, above]{$3$} (1.7,0.67) ;
	\path[->] (0,0) edge node[at end, above]{$2$} (0,1.5) ;
	\path[->] (0,0) edge node[at end, above]{$1$} (1.5,0) ;
	\path[->] (0,0) edge node[at end, below= 0.5mm]{$4$} (0.77, 1.04)  ;
	\path[->] (0,0) edge node[at end,  above]{$5$} (1,1.35) ;
\end{tikzpicture}
\end{subfigure}
\caption{}
\label{}
\end{figure}

Suppose $l_i$ is a non-loop in $M$ such that all elements in $P_M(l_i)$ are loops in $N_0$. Let $a$ be a non-loop in $N_0$ satisfying the condition that if for any element $p$ such that $\chi_M(p, a) = +$, then $\chi_M(p, l_i) = +$. Similarly, let $b$ be a non-loop in $N$ satisfying the condition that if for any element $p$ such that $\chi_M(b, p) = +$, then $\chi_M(l_i,p) = +$. The face of $\{(\theta_1, \theta_2, \ldots, \theta_n): 0 \leq \theta_1 \leq \theta_2 \leq \cdots \leq \theta_n \leq \frac{\pi}{2}\}$ corresponding to $\theta_a = \theta_{l_1} = \theta_{l_2} = \cdot \theta_{l_i} < \cdots < \theta_b$ determines a rank $2$ oriented matroid $N'$ such that $N_0 < N' < M$. Repeating the same argument for all such $l_i$ with all elements in $P_M(l_i)$ loops in $N$, we can assume that $N'$ is a rank $2$ oriented matroid  such that $N' < M$ is obtained from $M$ by sequences of \textbf{CR2} and each parallel class $P_{N'}(i)$ contains a non-loop of $N$. If $a, b$ are non-loops in $N$ such that $\chi_{N}(a, b) = 0$ and $\chi_N'(a, b) = +$, this corresponds to the face of $\{(\theta_1, \theta_2, \ldots, \theta_{n_1}): 0 \leq \theta_1 \leq \theta_2 \leq \cdots \leq \theta_n \leq \frac{\pi}{2}\}$ determined by $\theta_a = \cdots = \theta_b$. From $N'$ we can obtain by sequences of \textbf{CR2} a rank $2$ oriented $N''$ having the same number of distinct parallel class as $N$ and  the same number of non-loops as $M$. We can then obtain $N$ from $N''$ by sequence(s) of $\textbf{CR1}$. In particular, if $M$ covers $N_0$, then the proposition follows.

\end{proof}

\begin{cor}
	$\mbox{MacP}(2,n)$ is a ranked poset, with rank function $h(M) = \dim(\mu^{-1}(M)) = L_M + P_M - 4$.
\end{cor}

\begin{lemma}\label{le}
	If $N < M$ are rank 2 oriented matroids such that $M$ covers $N$. Then $\mu^{-1}(N) \subseteq \partial\overline{(\mu^{-1}(M))}$. 
\end{lemma}
\begin{proof}
	
\textbf{Case 1}: Suppose the covering relation $N < M$ is CR1. There is exactly one $p$ in a parallel/anti-parallel class $P$ of $M$ such that $|P|\geq 2$ in  $M$, $p$ is a loop in $N$ and $\chi_N(j,k) = \chi_M(j,k)$ for $j,k \neq i$.
	
Let $X_1 = \mbox{Rowspace}(v_1 \; v_2 \; v_3 \; \cdots v_k) \in \mu^{-1}(N)$. If $p$ is parallel to $j$ in $M$ and $p$ is a loop in $N$ that is $v_p = 0$, then $\mbox{Rowspace}(v_1 \; v_2 \; \cdots v_{p-1} \; \epsilon \cdot v_j \; v_{p+1} \ldots v_n) \in \mu^{-1}(M)$ for every $\epsilon > 0$. Hence $\mu^{-1}(N) \subseteq  \partial \overline{\mu^{-1}(M)}$
	
\textbf{Case 2}: Suppose the covering relation $N < M$ is CR2. There are two distinct parallel/anti-parallel classes $P_M(s)$ and $P_M(j)$ in $M$ such that the following holds: $$\chi_N(k,l) = \left\{ \begin{array}{ccc}
	0 & \mbox{if} & k,l \in P_M(s) \cup P_M(j)\\
	\chi_M(k,l) & & \mbox{otherwise}
	\end{array}\right\}$$\\
	Let $X_1 = \mbox{Rowspace}(e_1 \; e_2 \; \cdots \; v_n)$ be any vector arrangement in $\mu^{-1}(N)$. Let $v_s = r_se^{i\theta_s}$ , $v_j = r_je^{i\theta_s}$,  and suppose WLOG that $s$ and $j$ are parallel in $N$. Let $v_j^\epsilon = r_se^{i(\theta_s + \epsilon)}$. If $\chi_{M}(s,j) = +$, ($v_j^{-\epsilon}$ if $\chi_{M}(p,j) = -$), then $\mbox{Rowspace}(e_1 \; e_2 \; v_3\; \cdots\\ v_{j-1}, v_j^\epsilon, v_{j+1}, \; \cdots \; v_n) \in \mu^{-1}(M)$ for every sufficiently small value of $\epsilon$. Hence, $\mu^{-1}(N) \subseteq \partial \overline{\mu^{-1}(M)} .$	 
\end{proof}

	\begin{prop}\label{prop2}
		Let $M \in \mbox{MacP}(2,n)$. Then $\partial \overline{\mu^{-1}(M)} = \bigcup_{N < M} \mu^{-1}(N)$
	\end{prop}

	\begin{proof}
Let $\chi$ be a chirotope of $M$. Then $\overline{\mu^{-1}(M)} \subseteq \{\mbox{Rowspace}(v_1 \; v_2\; v_3\; \cdots \; v_n) : \mbox{sign}(\det(v_i \; v_j)) \in  \{0, \chi(i,j)\}\}$. Let $X \in \overline{\mu^{-1}(M)}$ and $N$ the rank $2$ oriented matroid determined by $X$. Then $N \leq M$.
	
Conversely,  suppose $N , M \in \mbox{MacP}(2,n)$ such that $N < M$. We will show that $\mu^{-1}(N) \subseteq  \overline{\mu^{-1}(M)}$.
	 
	Let  $N = N_1 < N_2 < \cdots < N_k = M$ be a maximal chain from $N$ to $M$. That $\mu^{-1}(N_i) \subseteq \partial \overline{\mu^{-1}(N_{i+1})}$ follows from Lemma ~\ref{le}.
	
From $\mu^{-1}(N_i) \subset \partial\overline{(\mu^{-1}(N_{i+1}))}$ for each $i$, we conclude that $\mu^{-1}(N) \subseteq \partial\overline{(\mu^{-1}(M))}$. 
\end{proof}	
A useful observation from Proposition \ref{cor1} and Proposition \ref{prop2} is the following:

Suppose $N < M$ are rank $2$ oriented matroids such that $P_M(l_1)$ and $P_M(l_2)$ are distinct parallel/anti-parallel classes in $M$, but $P_M(l_1) \cup P_M(l_2)$ is a parallel/anti-parallel class in $N$. Then we can obtain a rank $2$ oriented matroid $N' \leq M$ such that $N'$ covers $N$, $P_M(l_1)$ and $P_M(l_2)$ are in distinct parallel/anti-parallel class in $N'$. 

An example of such a construction is the following:

Suppose $\mbox{Rowspace}(v_1\; v_2 \; \ldots \; v_n)$ determines a vector realization for $N$ and assume WLOG that $P_N(l_1) = P_M(l_1) \cup P_M(l_2) $. We obtain a rank $2$ oriented matroid $N'$ as follows: let $v_{l_2} = r_{l_2}e^{i\theta_{l_2}}$ and,  we denote by  $\overline{P_M(l_1)}$ the subset of $P_M(l_1) $ that are anti-parallel to $l_2$ in $N$:
	$$w_s = \left\{\begin{array}{ccc}
	v_s & \mbox{if} & s \notin   P_M(l_1)\\
	r_{l_2}e^{i(\theta_{l_2} + \epsilon)} & \mbox{if} & s \in P_M(l_1)\setminus \overline{P_M(l_1)} , \chi_M(l_2, s) = +\\
	r_{l_2}e^{i(\theta_{l_2} - \epsilon)} & \mbox{if} & s \in P_M(l_1) \setminus \overline{P_M(l_1)}, \chi_M(l_2, s) = - \\
	r_{l_2}e^{i(\theta_{l_2} + \pi - \epsilon)} & \mbox{if} & s \in \overline{P_M(l_1)},\chi_M(l_2, s) = +\\
	r_{l_2}e^{i(\theta_{l_2} + \pi + \epsilon)} & \mbox{if} & s \in \overline{P_M(l_1)},\chi_M(l_2, s) = -\\
	\end{array}\right\}.$$
	The rank $2$ subspace $\mbox{Rowspace}(w_1\; w_2\; \ldots\; w_n)$ determines a rank $2$ oriented matroid $N' \leq M$ that covers $N$.

\section{Shellability of the interval $\mbox{MacP}(2,n)_{\leq M} \cup \{\hat{0}\}$} \label{sec:shell_1}
We will show that the poset $\mathrm{MacP}(2,n)_{\leq M} \cup \{\hat{0}\}$ is the augmented face poset of a regular cell decomposition of a $h(M)-1$ dimensional sphere.

\begin{lemma}\label{lem}
	Let $W < T$ be rank 2 oriented matroids. Then the interval $[W, T]$ is totally semimodular. 
\end{lemma}
\begin{proof}
	Let $N, N_1, N_2 \in [W, T]$ such that $N_2$ and $N_1$ cover $N$. In Proposition ~\ref{cor1}, we proved that there are two possible scenarios when $N_1$ covers $N$ and similarly for $N_2$ covering $N$. So there are the following three distinct cases:
	
	\textbf{Case 1:} If $N < N_1$ and $N < N_2$ are both case CR1 of Proposition ~\ref{cor1}. That is, there are $i\neq j$ such that $i,j$ are loops in $N$ but $i$ is a non-loop in $N_1$ with parallel/anti-parallel class $P_1(i)$ in $N_1$ such that $|P_1(i)| \geq 2$ in $N_1$. Also, that $j$ is a non-loop  with parallel/anti-parallel class $P_2(j)$ in $N_2$ so that $|P_2(j)| \geq 2$ and $j$ is a loop in $N$ . Let $k_1 \in P_1(i) \setminus \{i,j\}$ and $k_2 \in P_2(j) \setminus \{i,j\}$. Then $k_1, k_2$ are non-loops in $N$. If $\mbox{Rowspace}(v_1 \; v_2 \; \cdots \; v_n)$ determines a vector realization of $N$, then $\mbox{Rowspace}(v_1\; v_2 \; \cdots v_{i-1} \; \epsilon_1 \cdot v_{k_1} \; v_{i+1}\; \cdots v_{j-1} \; \epsilon_2\cdot v_{k_2}\; \cdots v_n)$ (where $\epsilon_i = \pm 1$) gives a rank $2$ oriented matroid $M$ that covers both $N_1$ and $N_2$.
	
	\textbf{Case 2:} If $N < N_1$ is case (a) of Proposition ~\ref{cor1} and $N < N_2$ is case $(b)$ of Proposition ~\ref{cor1}, that is, there is a non-loop $i$ with $|P_1(i)| \geq 2$ in $N_1$ and $i$ is a loop in $N$. Also, there are distinct parallel/ anti-parallel classes $P_2(j)$ and $P_2(k)$ in $N_2$ such that $P_2(j) \cup P_2(k)$ is a parallel/ anti-parallel class in $N$. Suppose WLOG that $t \in P_1(i)\setminus i$ is parallel to $i$ in $N_1$.
	
	We have that $i$ is a loop in $N_2$ since $N_2$ covers $N$. Now, if $\mbox{Rowspace}(v_1 \; v_2\;\ldots v_i \ldots\; v_n)$ determines a vector realization for $N_2$, then $\mbox{Rowspace}(v_1\; v_2 \; \ldots v_{i-1}\; v_t\; v_{i+1}\ldots \; v_n)$ determines a vector realization for $M$ that covers both $N_1$ and $N_2$.
	
	\textbf{Case 3:}If $N < N_1$ and $N < N_2$ are both case $(b)$ in Lemma ~\ref{cor1}. That is, there are distinct parallel/anti-parallel classes $P_1(l_1), P_1(l_2)$ in $N_1$ such that $P_1(l_1) \cup P_1(l_2)$ is a parallel/anti-parallel class in $N$. Similarly, there are classes $P_2(l_3), P_2(l_4)$ in $N_2$ such that $P_2(l_3) \cup P_2(l_4)$ is a parallel/anti-parallel class in $N$. 
	
	As in the observation after Proposition \ref{prop2}, we can obtain a rank $2$ subspace $\mbox{Rowspace}(w_1\; w_2\; \ldots\; w_n)$ that determines a rank $2$ oriented matroid $M$ that covers both $N_1$ and $N_2$.
\end{proof}
We have from the above lemma that for a rank $2$ oriented matroid $M$, every bounded interval in $\mbox{MacP}(2,n)_{\leq M} $ is totally semimodular.

Let $M$ be a rank $2$ oriented matroid, and let $N$ be a rank 2 oriented matroid in the interval $P = \mbox{MacP}(2,n)_{\leq M} \cup \{\hat{0}\}$. Then the interval $[N, M]$ is totally semimodular by  Lemma ~\ref{lem}. By Theorem ~\ref{bw}, any atom ordering in $[N, M]$ is a recursive atom ordering.

So to find a recursive atom ordering for $\mbox{MacP}(2,n)_{\leq M} \cup \{\hat{0}\}$, we only need to order the atoms of $\mbox{MacP}(2,n)_{\leq M} \cup \{\hat{0}\}$ as $X_1, X_2, \ldots , X_k$ so that the ordering satisfies conditions in Definition ~\ref{raon}.
\begin{prop}\label{rao}  
Let $M$ be a rank $2$ oriented matroid. Then the interval $\mbox{MacP}(2,n)_{\leq M} \cup \{\hat{0}\}$ has a recursive atom ordering.
	\end{prop}
	\begin{proof}
	We begin by ordering the atoms of $M$ as follows (see Figure ~\ref{rec} below): Fix a vector realization $ A = \{v_1, v_2, v_3, \ldots, v_n\}$ that determines $M$. The vector realization $A$  determines vector realizations for atoms of $M$. Let $L$ be an affine line not parallel to any of the lines spanned by the non-zero vectors in $\{v_i\}$. 
	
	We label the vectors in the realization $A$ in an increasing order from left to right in the order lines spanned by the vectors intersect the affine line $L$. We note that elements in the same parallel/anti-parallel class $P_M(i)$ have their linear spans intersect $L$ at the same point (in the Figure ~\ref{rec} below,\;  the pairs $1,2$ and $3,4$ labeled vectors are in a parallel/anti-parallel class respectively).
	 
	 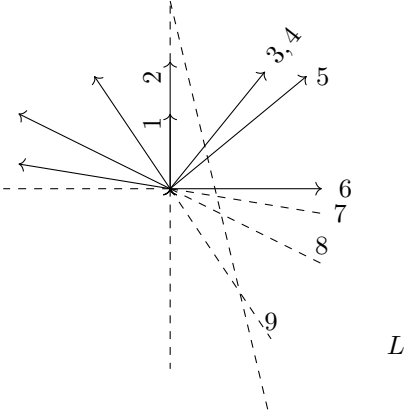
\begin{figure}[htb]
	 	\begin{tikzpicture}
	 	\draw[->] (0,0) edge node[very near end, sloped, above]{$2$} (0,1.7) ;
	 	\draw[dashed] (0,0) edge node[very near end, sloped, above]{} (0,-2.4) ;
	 	\draw[->] (0,0) edge node[very near end, sloped, above]{} (-1,1.5);
	 	\draw[dashed] (0,0) edge node[at end, above]{$9$} (1.33,-2) ;
	 	\draw[->] (0,0) edge node[very near end, sloped, above]{} (-2,1) ;
	 	\draw[dashed] (0,0) edge node[at end, above]{$8$} (2,-1) ;
	 	\path[dashed] (0,0) edge (-2.2, 0);
	 	\path[->] (0,0) edge node[very near end, sloped, above]{$1$} (0,1);
	 	\path[->] (0,0) edge node[at end, above , right = 1mm]{$6$} (2,0) ;
	 	\path[->] (0,0) edge node[sloped ,very near end, above , right = 3mm]{$3,4$} (1.25,1.56) ;
	 	\draw[->] (0,0) edge node[at end, above , right]{$5$} (1.8, 1.5);
	 	\draw[dashed] (0,1.7) edge (0,2.5);
	 	\path[dashed] (0,0) edge node[at end, right= 0.3mm]{$7$} (2,-0.33);
	 	\path[->] (0,0) edge node[at end, right= 0.3mm]{} (-2,0.33);
	 	\draw[dashed] (0,2.5) edge node[near end, right = 20mm, below = 10mm, above= 3mm]{$L$} (1.3,-3);
	 	\end{tikzpicture}\\
	 	\caption{Recursive atom ordering}
	 	\label{rec}
	 \end{figure}

An atom of $\mbox{MacP}(2,n)_{\leq M} \cup \{\hat{0}\}$ can be identified with $(i, j)$ for some $i < j$ in the above ordering, and the vectors labeled by $i,j$ form a basis for $M$. 
	  
	  We will also let $(i, jr)$ denote a rank $2$ oriented matroid in $\mbox{MacP}(2,n)_{\leq M} \cup \{\hat{0}\}$ covering the atoms $(i,j)$ and $(i,r)$, with element labeled $j$ parallel/anti-parallel to element labeled $r$. Similarly, we have a notation of the form $(ir, j)$.
	     
	  We consider a dictionary order for ordered pairs for the atoms of $M$ as $X_1, X_2,\ldots, X_k$. We now verify the conditions in Definition ~\ref{rao}. By Lemma ~\ref{lem}, the interval $[X_i, M]$ is totally semimodular for each $i$, and so it has a recursive atom ordering by Theorem ~\ref{bw}. In fact any ordering of the atoms in $[X_i, M]$ gives a recursive atom ordering of the poset $[X_i, M]$ by Theorem ~\ref{bw}.
	  
	  We now verify the second condition in Definition ~\ref{rao}. Suppose $A_1 = (i_1, j_1) < A_2 = (i_2, j_2)$ in the dictionary order and $A_1, A_2 < Y$ where $Y \in \mbox{MacP}(2,n)_{\leq M} \cup \{\hat{0}\}$. We will consider the following cases:
	  
	  \textbf{Case 1:} Suppose $i_1 = i_2$. Then $j_1 < j_2$. Let $j$ be the maximum label such that $j_1 \leq j < j_2$ and the element labeled by $j$ is a non-loop in $Y$. We obtain $Z$ as $(i_1, jj_2)$, where $j$ and $j_2$ are parallel if $\mathbb{R}_{\geq 0}(\epsilon_1 \cdot v_j), \mathbb{R}_{\geq 0}(\epsilon_2 \cdot v_{j_2})$ intersect $L$ for $\epsilon_i \in \{\pm 1\}$, $\epsilon_1\cdot \epsilon_2 =1$,  and anti-parallel if $\epsilon_1 \cdot \epsilon_2 = -1$. Such a $Z$ covers both $(i_1, j)$ and $(i_1, j_2)$. We have  $Z \leq Y$ and $(i_1, j) < (i_2, j_2)$ in the dictionary order.
	  
	  \textbf{Case 2:} Suppose $i_1 < i_2$. Let $i$ be the maximum label such that $i_1 \leq i < i_2$ and the element labeled $i$ is not a loop in $Y$. Then $Z$ is obtained as $(ii_2, j_2)$. Similarly, $i, i_2$ are parallel if $\mathbb{R}_{\geq 0}(\epsilon_1 \cdot v_i), \mathbb{R}_{\geq 0}(\epsilon_2 \cdot v_{i_2})$ intersect $L$ for $\epsilon_i \in \{\pm 1\}$, $\epsilon_1\cdot \epsilon_2 =1$,  and anti-parallel if $\epsilon_1 \cdot \epsilon_2 = -1$. Such a $Z$ is a rank $2$ oriented matroid covers the atoms $(i, j_2)$ and $(i_2, j_2)$. We have that $Z \leq Y$ and $(i, j_2) < (i_2, j_2)$ in the dictionary order.
	  
	  For each $j > 1$, let $Q_j = \{Y \in \mbox{atom}[X_j, M]: Y \geq X_i \; \mbox{for\; some\;} i  < j\}$. In the recursive atom ordering of $[X_j, M]$ for $j > 1$, we let elements of  $Q_j$ come first. This determine a recursive atom ordering for the poset $[X_i, M]$ by Theorem ~\ref{bw}, as the interval $[X_i, M]$ is totally semimodular for each $i$ by Lemma ~\ref{lem}.
	  
	  Hence the poset $\mbox{MacP}(2,n)_{\leq M} \cup \{\hat{0}\}$ has a recursive atom ordering.

\end{proof}

\begin{lemma}\label{thin}
 Let $N$ be a rank 2 oriented matroid. Then the interval $\mbox{MacP}(2,n)_{\leq N} \cup \{\hat{0}\}$ is thin.
 \end{lemma}

\begin{proof}
In the interval $\mbox{MacP}(2,n)_{\leq N} \cup \{\hat{0}\}$ when $h(N) = 1$ in $\mbox{MacP}(2,n)$, we know that $|\mbox{MacP}(2,n)_{\leq N} \cup \{\hat{0}\}| = 4$. So, we can consider intervals  $[N_0, N_2]$ of length $2$,  where $N_0$ is a rank $2$ oriented matroid. By the covering relations in Proposition ~\ref{cor1}, $3 \leq |[N_0, N_2]| \leq 4$. We will now prove that $|[N_0, N_2]| = 4$. Suppose $N_0 < N_1 < N_2$ is a maximal chain in $[N_0, N_2]$, we will consider the covering relations $N_0 < N_1$ and $N_1 < N_2$ according to Proposition ~\ref{cor1}.

\textbf{Case 1:} If the coverings $N_0 < N_1$ and $N_1 < N_2$ are both \textbf{CR1} of Proposition ~\ref{cor1}. That is, there are non-loops $i, j$ so that the parallel/anti-parallel class $P_1(i)$ has size $|P_1(i)| \geq 2$ in $N_1$ and $i$ is a loop in $N_0$. Also, the parallel/anti-parallel class $P_2(j)$ has size $|P_2(j)|\geq 2$ in $N_2$ and $j$ is a loop in $N_1$. Now, suppose $\mbox{Rowspace}(v_1\; v_2 \ldots \; v_n)$ is determines a vector realization for $N_2$. Then $\mbox{Rowspace}(v_1 \; v_2 \; \ldots v_{i-1} , 0 , v_{i+1}\; \ldots v_n)$ determines a vector realization for a rank $2$ oriented matroid $N_1'\neq N_1$ so that  $N_0 < N_1' < N_2$.

\textbf{Case 2:} If the covering $N_0 < N_1$  is \textbf{CR1} of Proposition ~\ref{cor1} and the covering $N_1 < N_2$ \textbf{CR2} of Proposition ~\ref{cor1}. That is, there is a non-loop $i$ in $N_1$ with $|P_1(i)| \geq 2$ in $N_1$ so that $i$ is a loop in $N_0$. 

If the size of parallel/anti-parallel class $P_2(i)$ is of size $|P_2(i)|\geq 2$ in $N_2$, then we obtain $N_1'$ as in the above case from $N_2$.

If $|P_2(i)| = 1$ in $N_2$, then there are distinct  classes $P_2(j), P_2(k) \neq P_2(i)$ in $N_2$ so that $P_2(i) \cup P_2(j)$ is a parallel/anti-parallel class in $N_1$; and if $\mathrm{Rowspace}(v_1 \; v_2 \; v_3 \; \cdots\\ v_k \cdots v_i \cdots  \; v_n)$ is a vector realization for $N_2$, then $N_2$ covers the rank two oriented matroid $N_1'$ obtained from $\mathrm{Rowspace}(v_1 \; v_2 \; v_3 \; \cdots v_k \cdots v_{i-1}\; v_k \; v_{i+1} \cdots  \; v_n)$. So that $N_0 < N_1' < N_2$ and $N_1 \neq N_1'$.

\textbf{Case 3:} If the covering $N_0 < N_1$ is \textbf{CR2} of Proposition ~\ref{cor1} and the covering $N_1 < N_2$ is \textbf{CR1} of Proposition ~\ref{cor1}. That is, there is a non-loop $k$ in $N_2$ so that $|P_2(k)| \geq 2$ and $k$ is a loop in $N_1$. Let $j \in P_2(k)\setminus \{k\}$ and $X_0= \mbox{Rowspace}(w_1\; w_2 \ldots \; w_n)$ determines a vector realization for $N_0$. Then $w_k = 0$. We obtain the vector realization $X_1'$ that determines $N_1'$ by replacing $w_k = 0$ with $w_k' =  w_j$. So that $N_0 < N_1' < N_2$ and $N_1 \neq N_1'$.


\textbf{Case 4:} If the coverings $N_0 < N_1$  and $N_1 < N_2$ are both \textbf{CR2} of Proposition ~\ref{cor1}. That is, there are classes $P_1(l), P_1(j)$ in $N_1$ so that $P_1(l) \cup P_1(j)$ is a parallel/anti-parallel class in $N_0$ and there are classes $P_2(k), P_2(r)$ in $N_2$ so that $P_2(k)\cup P_2(r)$ is a parallel/anti-parallel class in $N_1$. Let $X_0 = \mbox{Rowspace}(v_1 \; v_2 \ldots \; v_n)$ determines a vector realization for $N_0$,  and $v_k = r_ke^{i\theta_k}$.

We can similarly obtain a vector $X'_1 = \mbox{Rowspace}(w_1\; w_2 \ldots \; w_n)$ as in the observation after Proposition \ref{prop2} so that $N_1'$ the rank $2$ oriented matroid determined by $X_1'$ is such that $N_0 < N_1' < N_2$ and $N_1' \neq N_1$.
	
\end{proof}

\begin{prop} \label{shell}
Let $M$ be a rank $2$ oriented matroid. Then the interval $\mathrm{MacP}(2,n)_{\leq M} \cup \{\hat{0}\}$ is the augmented face poset of a regular cell decomposition of a $PL$ sphere.
\end{prop}
\begin{proof}[Proof of Proposition \ref{shell}]
	The proposition now follows from Proposition ~\ref{rao}, Lemma ~\ref{thin} and Theorem \ref{bjorner}. 
\end{proof}

\section{The topology of $\overline{\mu^{-1}(M)}$}\label{sec:top_1}

To prove that $\overline{\mu^{-1}(M)}$ is homeomorphic to a closed ball, we will first prove that $\overline{\mu^{-1}(M)}$ is a topological manifold with boundary $\bigcup_{N < M} \mu^{-1}(N)$ as suggested by the following theorems and propositions.

\begin{thm}\label{top}
 Let $M \in \mbox{MacP}(2,n)$. The closure $\overline{\mu^{-1}(M)}$ is a  topological manifold whose boundary is $\bigcup_{N < M} \mu^{-1}(N)$.
 \end{thm}
\begin{proof}
	That the boundary of $\overline{\mu^{-1}(M)} = \bigcup_{N < M} \mu^{-1}(N)$ was obtained in Proposition \ref{prop2}.
	
	For every $X \in \partial\overline{\mu^{-1}(M)}$, we will obtain a closed neighborhood $N_X$ of $X$ that is homeomorphic to a closed ball of dimension $h(M)$ and such that $X$ is a point on the boundary of $N_X$.
	
	Let $N = \mu(X)$. We may assume that $\{1,2\}$ is a basis of $N$. So, $X$ can be obtained as $X = \mbox{Rowspace}(e_1 \; e_2 \; v_1 \; v_2 \cdots \; v_{n-2}) \in  \partial\overline{\mu^{-1}(M)}$.

	We recall from $\textbf{ID1}$ preceding Proposition ~\ref{prop1} the following identification
	$$\mu^{-1}(M) \cong \mbox{int}(\{(x_1, x_2, \ldots, x_{n_1}) : 0 \leq x_1 \leq x_2 \cdots \leq x_{n_1} \leq\frac{\pi}{2}\}) 
	\times $$ $$\mbox{int}(\{(y_1, y_2, \ldots, y_{n_2}) : \frac{\pi}{2} \leq y_1 \leq y_2 \leq \cdots \leq y_{n_2}\leq\pi \}) \times (0, \infty)^{L_M-2}.$$
	
	Let $r_j = |v_j|$ for $j \geq 3$. We define a closed ball of dimension $\dim(\mu^{-1}(M))$ as follows: 
	 $$ B_M =  \left \{(x_1, x_2, \ldots, x_{n_1}) : 0 \leq x_1 \leq x_2 \cdots \leq x_{n_1} \leq\frac{\pi}{2}\right \} 
	 \times$$  $$\left\{(y_1, y_2, \ldots, y_{n_2}) : \frac{\pi}{2} \leq y_1 \leq y_2 \leq \cdots \leq y_{n_2}\leq\pi \right\} \times  \prod_{j=1}^{L_M-2} [\frac{r_j}{2}, r_j+1].$$
	 
	 There is a map $T: B_M \rightarrow \overline{\mu^{-1}(M)}$ defined as follows:
	 
	 $\left((\theta_1, \theta_2, \ldots \theta_{n_1}), (\theta_{n_1+1}, \theta_{n_1+2}, \ldots \theta_{n_1+n_2}), (r_j)_j\right) \mapsto \mathrm{Rowspace}(e_1 \; e_2 \; w_1 \; w_2 \; \cdots w_{n-2})$, where $w_j= r_je^{i\theta_j}$. $X$ is contained in the image of $T$.
	 
	 The map $T$ is not necessarily one-to-one on the boundary of $B_M$. Figure \ref{viz_0} illustrates two points on the boundary $\{r_3 = 0, r_6 = 0\}$ of $B_M$ with the same image under $T$.
	 
	 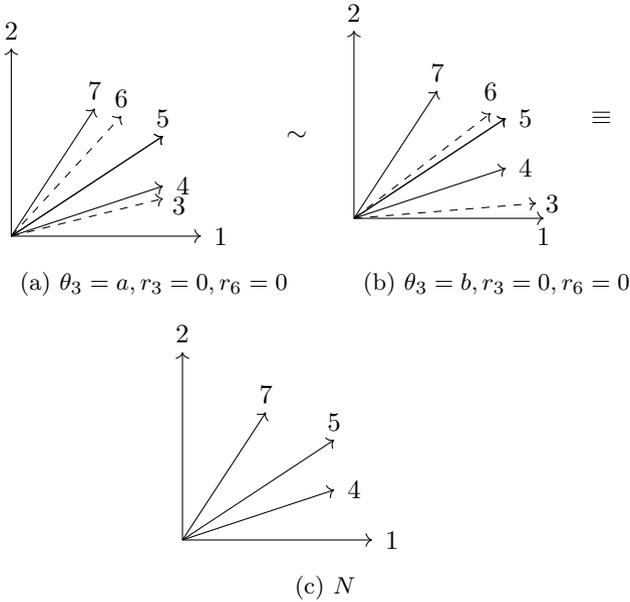
\begin{figure}[!htb]
	 	\centering
	 \begin{subfigure}[t]{0.35\textwidth}
	 	\begin{tikzpicture}
	 	\path[->] (0,0) edge node[at end, right = 0.5mm]{$4$} (3/1.5,1/1.5) ;
	 	\path[->] (0,0) edge node[at end, above]{$2$} (0,2.5) ;
	 	\path[->] (0,0) edge node[at end, right= 0.5mm]{$1$} (2.5,0) ;
	 	\path[->] (0,0) edge node[at end, above]{$5$} (3/1.5,2/1.5)  ;
	 	\path[->] (0,0) edge node[at end, above]{$7$} (1.1, 1.7) ;
	 	\path[dashed][->] (0,0) edge node[at end, below=1mm, right]{$3$} (2.0, 0.50) ;
	 	\path[dashed][->] (0,0) edge node[at end, above]{$6$} (1.45, 1.6);
	 	\path[->] (0,0) edge node[at end, above,  right = 15mm]{$\sim$} (3/1.5, 2/1.5);
	 	\end{tikzpicture}
	 	\caption{$\theta_3 = a, r_3 = 0, r_6 = 0$}
	 	\end{subfigure}
	 ~
	 \begin{subfigure}[t]{0.35\textwidth}
	 	\begin{tikzpicture}
	 	\path[->] (0,0) edge node[at end, right = 0.5mm]{$4$} (3/1.5,1/1.5) ;
	 	\path[->] (0,0) edge node[at end, above]{$2$} (0,2.5) ;
	 	\path[->] (0,0) edge node[at end, below]{$1$} (2.5,0) ;
	 	\path[->] (0,0) edge node[at end, above, right = 0.5mm]{$5$} (3/1.5,2/1.5)  ;
	 	\path[->] (0,0) edge node[at end,above]{$7$} (1.1, 1.7) ;
	 	\path[dashed][->] (0,0) edge node[at end, above = 0.5mm ]{$6$} (1.8, 1.4);
	 	\path[->] (0,0) edge node[at end, above,  right = 10mm]{$\equiv$} (3/1.5, 2/1.5);
	 	\path[dashed][->] (0,0) edge node[at end, right]{$3$} (2.4, 0.2);
	 	\end{tikzpicture}
	 	\caption{$\theta_3 = b, r_3 = 0, r_6 = 0$}
	 	\end{subfigure}
 	 ~
 	\begin{subfigure}[t]{0.35\textwidth}
 		\begin{tikzpicture}
 		\path[->] (0,0) edge node[at end, right = 0.5mm]{$4$} (3/1.5,1/1.5) ;
 		\path[->] (0,0) edge node[at end, above = 0.07mm]{$2$} (0,2.5) ;
 		\path[->] (0,0) edge node[at end, right= 0.5mm]{$1$} (2.5,0) ;
 		\path[->] (0,0) edge node[at end, above]{$5$} (3/1.5,2/1.5)  ;
 		\path[->] (0,0) edge node[at end,above]{$7$} (1.1, 1.7) ; 
 		\end{tikzpicture}
 		\caption{$N$}
 	\end{subfigure}
	 	\caption{An identification on the boundary of $B_M$.}
	 	\label{viz_0}
	 \end{figure}

Let $\sim$ be such identifications on the boundary of $B_M$ given by $x \sim y$ if and only if $T(x) = T(y)$. The identification is as discussed in Example \ref{ex}.  
	 
The map $T' = T/\sim \; : B_M/\sim  \; \rightarrow \overline{\mu^{-1}(M)}$ is injective. By Theorem ~\ref{hersh}, $B_M/\sim $ is homeomorphic to $B_M$. The image $T'(B_M/\sim)$ is homeomorphic to $B_M/\sim$ by the compactness of $B_M/\sim$ and $T'(B_M/\sim) \subset \overline{(\mu^{-1}(M))}$ being Haursdorff  . 
	 
So, the desired closed neighborhood of $X$ in $\overline{\mu^{-1}(M)}$ is $T'(B_M/\sim)\equiv B_M/\sim$ and it is homeomorphic to $B_M$ a closed ball of dimension $\dim(\mu^{-1}(M))$. Hence, $\overline{\mu^{-1}(M)}$ is a topological manifold with boundary.

\end{proof}

The following theorem will now follow from  Proposition ~\ref{prop2},  Proposition ~\ref{rao}, Lemma ~\ref{thin}, Theorem ~\ref{bjorner} and Theorem ~\ref{top}. 

\begin{thm}\label{cball}
	Let $M$ be a rank $2$ oriented matroid. There is an homeomorphism from $\overline{\mu^{-1}(M)}$ to an $h(M)$-dimensional closed ball. 
\end{thm}
\begin{proof}
The proof proceeds by induction on the height of $M$ in $\mbox{MacP}(2,n)$. The theorem is true for $h(M) = 0$ in which case $\overline{\mu^{-1}(M)}$ is a $0$-dimensional ball. Suppose $h(M) = l$, $l\geq 1$, and assume the theorem is true for every $N$ in $\mathrm{MacP}(2,n)$ with $h(N) < l$, that is, $\overline{\mu^{-1}(N)}$ is a $h(N)$-dimensional closed ball. From Proposition ~\ref{prop2}, we know that $\partial\overline{(\mu^{-1}(M))} = \bigcup_{N < M}(\mu^{-1}(N))$, and by our inductive assumption $\bigcup_{N < M}(\mu^{-1}(N))$ is a regular cell complex with cells $\{\mu^{-1}(N) : N < M\}$.

Proposition ~\ref{rao} says that the interval $[\hat{0},M]$ has a recursive atom ordering, and the interval is thin by Lemma ~\ref{thin}. By Theorem ~\ref{bjorner}, $\mathrm{MacP}(2,n)_{\leq M} \cup \{\hat{0}\}$ is an augmented poset of a regular cell decomposition of a $h(M)-1$-dimensional sphere. Hence, the regular cell complex $\partial\overline{(\mu^{-1}(M))} = \bigcup_{N < M}(\mu^{-1}(N))$ is homeomorphic to a sphere of dimension $h(M)-1$.

So, the boundary of $\mu^{-1}(M)$ is homeomorphic to a $(h(M) -1)$ sphere.  By Theorem ~\ref{top} , we now know that  $\overline{\mu^{-1}(M)}$ is a topological manifold; $\overline{\mu^{-1}(M)} $ is a topological manifold whose boundary $\bigcup_{N < M}(\mu^{-1}(N))$ is a sphere of dimension $h(M)-1$, and its interior $\mu^{-1}(M)$ is an open ball by Proposition ~\ref{prop1}.

Hence,  by  Proposition ~\ref{ball}, we obtain an homeomorphism from $\overline{\mu^{-1}(M)}$ to a closed $h(M)$-dimensional ball.
\end{proof}

\begin{proof}[Proof of Theorem ~\ref{main}]
That $\{\mu^{-1}(M) : M \in \mbox{MacP}(2,n)\}$ forms a regular cell decomposition of $Gr(2,\mathbb{R}^n)$ follows from Theorem \ref{cball}. Hence, $\|\mathrm{MacP}(2,n)\|$ is homeomorphic to $\mathrm{Gr}(2, \mathbb{R}^n)$.
\end{proof}

\section{The stratification $\nu^{-1}(N, M)$} \label{sec:ball_2}
As described in the background section, each point $(Y , X)$ in $Gr(1,2,\mathbb{R}^n)$ can be identified with a point in $\mathrm{Gr}(2,n)$ as illustrated in Figure \ref{fig4}. In the Figure, we may assume that $\{1,2\}$ is a basis of $M = \mu(X)$, so that $(Y, X)$ is identified with $\mathrm{Rowspace}(e_1 \; e_2 \; v_3 \; v_4 \cdots \; v_{n+1})$

\begin{figure}[!htb]
	\centering
\begin{subfigure}[t]{0.35\textwidth}
	\begin{tikzpicture}
\path[->] (0,0) edge node[at end, right=0.5mm]{$3$} (3,1) ;
\path[->] (0,0) edge node[at end, above = 0.5mm]{$2$} (0,3) ;
\path[->] (0,0) edge node[at end, right= 0.5mm]{$1$} (3,0) ;
\path[->] (0,0) edge node[at end, right = 0.5mm]{$4$} (3,2)  ;
\path[->] (0,0) edge node[at end, right = 0.5mm]{$5$} (3,4) ;
\draw[dashed, ->] (0,0) edge node[at end , right = 0.5mm]{$v_{n+1}$} (2.8, 2.8);
	\end{tikzpicture}
	\caption{$h_M(N) = 1$}
\end{subfigure}
~
\begin{subfigure}[t]{0.35\textwidth}
\begin{tikzpicture}
\path[->] (0,0) edge node[at end, below=0.25mm]{$3$} (3,1) ;
\path[->] (0,0) edge node[at end, above = 0.5mm]{$2$} (0,3) ;
\path[->] (0,0) edge node[at end, right= 0.5mm]{$1$} (3,0) ;
\path[->] (0,0) edge node[at end, right = 0.5mm]{$4$} (3,2);
\path[->] (0,0) edge node[at end, right = 0.5mm]{$5$} (3,4) ;
\draw[dashed, ->] (0, 0) edge node[at end, right = 0.5mm]{$v_{n+1}$} (4.2, 1.4);

\end{tikzpicture}
\caption{$h_M(N) = 0$}
\end{subfigure}
\caption{Elements in $\mathrm{Gr}(1,2, \mathbb{R}^n).$}
\label{fig4}
\end{figure}
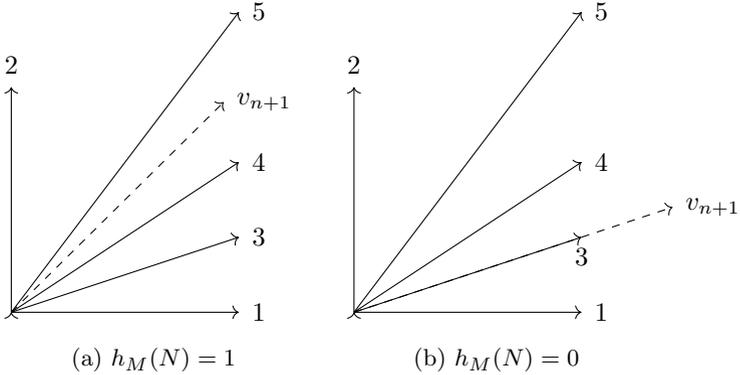

  \begin{nota}
  Let $M$ be a rank $2$ oriented matroid and $N \in G(1, M)$ a rank $1$ strong map image of $M$. The height of $N$ in the poset $G(1, M)$ is denoted by $h_M(N)$.\\
 \end{nota}
 
 \begin{thm}\label{main1}
 	$\{\nu^{-1}(N ,M) : (N, M) \in \mathrm{MacP}(1,2,n)\}$ forms a regular cell decomposition of $Gr(1,2,\mathbb{R}^n)$. 
 \end{thm}
 
Let $(Y, X) \in \mathrm{Gr}(1,2, \mathbb{R}^n)$, $N = \mu(Y)$ and $M = \mu(X)$. By the unique identification of a point $(Y, X) \in \mathrm{Gr}(1,2, \mathbb{R}^n)$ with the point $\mathrm{Rowspace}(e_1\; e_2 \; v_3 \; v_4 \; \cdots \; v_{n+1} )$ satisfying $0 \leq \mathrm{Arg}(v_j) < \pi$ for $j \geq 3$ and $\|v_{n+1}\| = 1$, we have the following identification of $\nu^{-1}(N, M)$:

By (\textbf{ID1}), $\mu^{-1}(M)$ is identified with 

$\mu^{-1}(M) \equiv A_M = \mathrm{Interior}(\{(\theta_{l_1}, \theta_{l_2}, \ldots, \theta_{l_{n_1}}) : 0 \leq \theta_{l_1} \leq \theta_{l_2} \cdots \leq \theta_{l_{n_1}} \leq\frac{\pi}{2}\} \times \{(\theta_{j_1}, \theta_{j_2}, \ldots, \theta_{j_{n_2}}) : \frac{\pi}{2} \leq \theta_{j_1} \leq \theta_{j_2} \leq \cdots \leq \theta_{j_{n_2}}\leq\pi \} \times [0, \infty)^{L_M-2})$. Then $\nu^{-1}(N, M)$ can thus be identified with 

$(\textbf{ID2}) \; \; \nu^{-1}(N, M) \equiv A_{(N,M)} = \mathrm{Interior}(\{(\theta_{l_1}, \theta_{l_2}, \ldots, \Theta,  \ldots , \theta_{l_{n_1}}) : 0 \leq \theta_{l_1} \leq \theta_{l_2} \cdots  \leq \Theta \leq \cdots \leq  \theta_{l_{n_1}} \leq\frac{\pi}{2}\} \times \{(\theta_{j_1}, \theta_{j_2}, \ldots, \theta_{j_{n_2}}) : \frac{\pi}{2} \leq \theta_{j_1} \leq \theta_{j_2} \leq \cdots \leq \theta_{j_{n_2}}\leq\pi \} \times [0, \infty)^{L_M-2})$

The following proposition thus follows from the identification of $\nu^{-1}(N, M)$ given in (\textbf{ID2}).

 \begin{prop}\label{homeo}
 	Let $(N, M) \in \mbox{MacP}(1,2,n)$. Then $\nu^{-1}(N, M)$ is homeomorphic to an open ball of dimension $h(M) + h_M(N)$.  
 \end{prop}

 \begin{lemma} (\cite{And:Jim})\label{max}
 	Let $M_1, M_2 \in \mbox{MacP}(2,n)$ and suppose that $M_1 < M_2$. Let $z_2 \in \mathcal{V}^*(M_2) \setminus \{0\}$. Then the set $A = \{z \in \mathcal{V}^*(M_1) \setminus \{0\} : z \leq z_2 \}$ has a maximal element. 
 \end{lemma}

\begin{nota}
 We will denote by $z_2^*$ the maximal non-zero covector in $\mathcal{V}^*(M_1)$ obtained above.
 \end{nota}

The following proposition follows as in Proposition \ref{prop2} applied to elements  $M' \in \mathrm{MacP}(2,n+1)$, and by the identification of the interval  $\mathrm{MacP}(1,2,n)_{\geq (N_0, M_0)}$ with the interval $\mathrm{MacP}(2, n+1)_{\geq M_0'}$ in (\textbf{SUBP}).
 
 \begin{prop}\label{bdry1}
 	For any $(N_0, M_0) \in \mbox{MacP}(1,2,n)$, we have
 	$$\partial\overline{(\nu^{-1}(N_0,M_0))} = \bigcup_{(N, M) < (N_0,M_0)} \nu^{-1}(N,M).$$  
 \end{prop}

\begin{thm}
	$\{\nu^{-1}(N,M) : (N, M) \in \mathrm{MacP}(1,2,n) \}$ forms a regular cell decomposition for $Gr(1,2,\mathbb{R}^n)$, where  $\nu : Gr(1,2, \mathbb{R}^n) \rightarrow \mbox{MacP}(1,2,n)$ is the map described in section 1.
\end{thm}

\section{Shellability of interval $\mathrm{MacP}(1,2,n)_{\leq (\pm z, M)}$}\label{sec:shell_2}

We prove the analogue of Lemma \ref{lem} here. We again use the identification of  $\mathrm{MacP}(1,2,n)_{(N_0, M_0)}$ with the interval $\mathrm{MacP}(2, n+1)_{\geq M_0'}$. An interval $[(N_1, M_1), (N_2, M_2)]$ in $\mathrm{MacP}(1,2,n)_{\geq (N_0, M_0)}$ can be identified with the interval $[M_1', M_2']$ in $\mathrm{MacP}(2, n+1)_{\geq M_0'}$ where $M_1', M_2'$ are the rank $2$ oriented matroid on $n+1$ elements corresponding to $M_1$ and $M_2$ respectively in \textbf{SUBP}.

\begin{lemma}\label{lemod}
Let $[(N_1, M_1), (N_2, M_2)]$ be an interval in $\mathrm{MacP}(1,2,n)$. Then the interval $[(N_1, M_1), (N_2, M_2)]$ is totally semimodular. 
\end{lemma}

\begin{prop}\label{rec1}
	 Let $(N ,M) \in Gr(1,2,\mathbb{R}^n)$. Then the interval $\mathrm{MacP}(1,2,n)_{\leq (N, M)} \cup \{\hat{0}\}$ has a recursive atom ordering.  
\end{prop}
\begin{proof}
	 We proceed as in Proposition ~\ref{rao}. Let $(Y, X) \in \nu^{-1}(N, M)$. We fix a vector arrangement $(e_1 \; e_2 \; v_3 \; v_4 \cdots v_{n+1})$ such that $(Y, X)$ is identified with $\mathrm{Rowspace}(e_1 \; e_2 \; v_3 \; v_4 \cdots v_{n+1})$. As in Proposition \ref{rao}, we also fix an affine line $L$ not parallel to any of the vectors in the above arrangement.
	 
	 Let $A_1, A_2, \ldots A_k$ be the recursive atom ordering of atoms of $\mathrm{MacP}(2,n)_{\leq M} \cup \{\hat{0}\}$ as in Proposition ~\ref{rao}. We denote by $L_z$ the labeling on the element $n+1$ as in Figure ~\ref{rc1} . 
	 
	 \begin{figure}[htb]
	 	\centering
	 	\begin{tikzpicture}
	 	\path[->] (0,0) edge node[at end, right]{$5$} (3,1) ;
	 	\path[->] (0,0) edge node[at end, left]{$1$} (0,3) ;
	 	\path[->] (0,0) edge node[at end, below= 1mm]{$6$} (3,0) ;
	 	\path[->] (0,0) edge node[at end, right]{$4$} (3,2)  ;
	 	\path[->] (0,0) edge node[at end, right]{$2$} (3,4) ;
	 	\draw[dashed] (0,0) edge node[at end, right]{$L_z= 3$} (3, 3);
	 	\draw[dashed, ->] (0,0) edge node[at end, right]{$v_{n+1}$} (2.5, 2.5);
	 	\draw[dashed] (0,0) edge  (2,2);
	 	\path[dashed] (0, 3) edge (0,3.5);
	 	\path[dashed] (3,0) edge (3.5,0);
	 	\path[dashed] (-0.5, 4) edge (4, -0.5); 
	 	\end{tikzpicture}
	 	\caption{}
	 	\label{rc1}
	 \end{figure}

Let $A_t = (i,j)$ be an atom of $\mathrm{MacP}(2,n)_{\leq M} \cup \{\hat{0}\}$. Let $z$ be a non-zero covector of $N$. We denote by $z_t^* = \mbox{max}\{w \in \mathcal{V}^*(A_t) \setminus 0: w \leq z\}$. Let $N^*_t$ be a rank $1$ oriented matroid with covectors $\{0, \pm z_t^*\}$. Then $(N_t^*, A_t) \leq (N, M)$. 

If $z_t^*$ is not a cocircuit of $A_t$, then there are cocircuits $z_t^0$, and $z_t^1$ satisfying $z_t^0(i) = 0$ and $z_t^1(j) = 0$ respectively. Let $N_t^0$ be a rank $1$ oriented matroid with covectors $\{0, \pm z_t^0\}$, and let $N_t^1$ be a rank $1$ oriented matroid with covectors $\{0, \pm z_t^1\}$. Then $(N_t^0, A_t)$ and $(N_t^1, A_t)$ are atoms in $\mathrm{MacP}(1,2,n)_{\leq (N, M)} \cup \{\hat{0}\}$ and are covered by $(N^*_t, A_t)$. If $z_t^*$ is a cocircuit of $A_t$, then $(N^*_t, A_t)$ is an atom in $\mathrm{MacP}(1,2,n)_{\leq (N, M)} \cup \{\hat{0}$. In this case, we say that $N^*_t = N_t^0 = N_t^1$.

In the case when when $z_t^*$ is not a cocircuit of $A_t$, we decide which of the atoms $(N_t^0, A_t)$ and $(N_t^1, A_t)$ is ordered just before the other. We consider the labeling obtained from the affine line $L$ as in Figure \ref{rc1}.

If $L_z \leq  i < j$ we order $(N_t^0, A_t)$ just before $(N_t^1, A_t)$ (see Figure \ref{ord1}). If $i < L_z \leq j$ we order $(N_t^1, A_t)$ just before $(N_t^0, A_t)$ (see Figure \ref{ord2}). If $L_z > j$ we order $(N_t^1, A_t)$ just before $(N_t^0, A_t)$ (see Figure \ref{ord3}). In the ordering of atoms of $\mathrm{MacP}(1,2,n)_{\leq (N_t^*, A_t)} \cup \{\hat{0}$, we denote the atom ordered first by $(N_t, A_t)$ while the second is denoted by $(N_t', A_t)$.

We then order the atoms of $\mathrm{MacP}(1,2,n)_{\leq (N, M)}$ as $(N_1, A_1), (N_1',A_1),\\
(N_2, A_2), (N_2', A_2) , \ldots 
(N_k, A_k), (N_k, A_k)$.

For each $(N_t, A_t)$ or $(N_t', A_t)$ in the above list, the intervals $[(N_t, A_t), (N, M)]$ and $[(N_t', A_t), (N, M)]$  are totally semimodular by Lemma \ref{lemod}. So by Theorem \ref{bw} the intervals $[(N_t, A_t), (N, M)]$, $[(N_t', A_t), (N, M)]$  have recursive atom ordering; in fact, any atom ordering gives a recursive atom ordering for the interval.

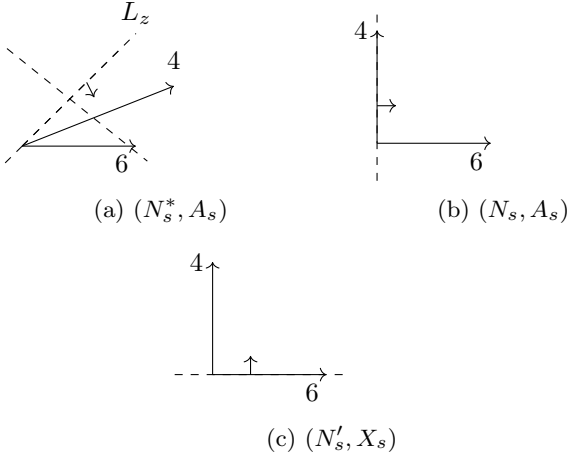
\begin{figure}[!htb]
	
	\centering
	
	\begin{subfigure}[t]{0.35\textwidth}
		\begin{tikzpicture}
		\path[->] (0,0) edge node[at end, above=1mm]{$4$} (2.0,0.8) ;
		\path[->] (0,0) edge node[very near end, below]{$6$} (1.5,0) ;
		\draw[dashed] (0,0) edge node[at end, above]{$L_z$} (1.5, 1.5);
		\draw[dashed] (0,0) edge node[at end ,above = 2mm]{} (1.3,1.3);
		\draw[dashed] (0,0) edge (-0.25, -0.25);
		\draw[dashed] (-0.2, 1.3) edge (1.65, -0.2);
		\path[->] (0.85, 0.85) edge (0.95, 0.65);
		\end{tikzpicture}
		\caption{$(N_s^*, A_s)$}
	\end{subfigure}
	~
	\begin{subfigure}[t]{0.35\textwidth}
		\begin{tikzpicture}
		\path[->] (0,0) edge node[at end, left]{$4$} (0,1.5) ;
		\path[->] (0,0) edge node[very near end, below]{$6$} (1.5,0) ;
		\draw[dashed] (0,-0.5) edge node[very near end, above]{} (0,1.80);
		\path[->] (0,0.5) edge (0.25, 0.5);
		\end{tikzpicture}
		\caption{$(N_s, A_s)$}
	\end{subfigure}
	~
	\begin{subfigure}[t]{0.35\textwidth}
		\begin{tikzpicture}
		\path[->] (0,0) edge node[at end, left]{$4$} (0,1.5) ;
		\path[->] (0,0) edge node[very near end, below]{$6$} (1.5,0) ;
		\draw[dashed] (-0.5,0) edge node[at end, above]{} (1.75,0);
		\path[->] (0.5,0) edge (0.5, 0.25);
		\end{tikzpicture}
		\caption{$(N_s', X_s)$}
	\end{subfigure}
	\caption{$L_z \leq  L_4 < L_6$}
	\label{ord1}
\end{figure}
  
 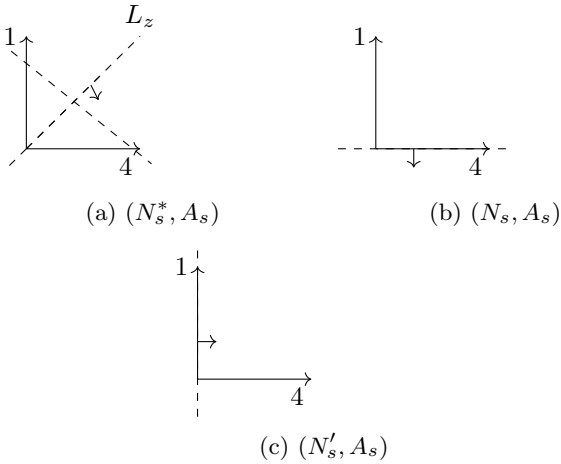
\begin{figure}[htb]
 	
 	\centering
\begin{subfigure}[t]{0.35\textwidth}
 	\begin{tikzpicture}
 	\path[->] (0,0) edge node[at end, left]{$1$} (0,1.5) ;
 	\path[->] (0,0) edge node[very near end, below]{$4$} (1.5,0) ;
 	\draw[dashed] (0,0) edge node[at end, above]{$L_z$} (1.5, 1.5);
 	\draw[dashed] (0,0) edge  (1,1);
 	\draw[dashed] (0,0) edge (-0.25, -0.25);
 	\draw[dashed] (-0.2, 1.3) edge (1.65, -0.2);
 	\path[->] (0.85, 0.85) edge (0.95, 0.65);
 	\end{tikzpicture}
 	\caption{$(N^*_s, A_s)$}
\end{subfigure}
~
\begin{subfigure}[t]{0.35\textwidth}
 	\begin{tikzpicture}
 	
 	\path[->] (0,0) edge node[at end, left]{$1$} (0,1.5) ;
 	\path[->] (0,0) edge node[very near end, below]{$4$} (1.5,0) ;
 	\draw[dashed] (-0.5,0) edge node[at end, above]{} (1.75,0);
 	\path[->] (0.5,0) edge (0.5, -0.25);
 	\end{tikzpicture}
 	\caption{$(N_s, A_s)$}
 	\end{subfigure}
 ~
 \begin{subfigure}[t]{0.35\textwidth}
 	\begin{tikzpicture}
 	\path[->] (0,0) edge node[at end, left]{$1$} (0,1.5) ;
 	\path[->] (0,0) edge node[very near end, below]{$4$} (1.5,0) ;
 	\draw[dashed] (0,-0.5) edge node[very near end, above]{} (0,1.75);
 	\path[->] (0,0.5) edge (0.25, 0.5);
 	\end{tikzpicture}
 	\caption{$(N_s', A_s)$}
 	\end{subfigure}
	\caption{$L_1 < L_z \leq L_4 $}
	\label{ord2}
 \end{figure}

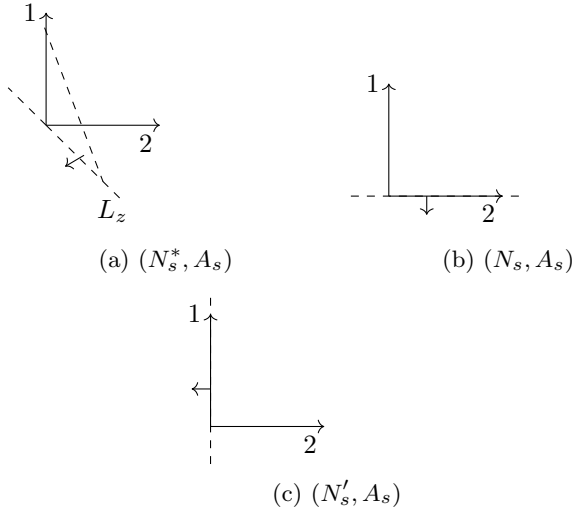
\begin{figure}[htb]
	
	\centering
	\begin{subfigure}[t]{0.35\textwidth}
		\begin{tikzpicture}
		\path[->] (0,0) edge node[at end, left]{$1$} (0,1.5) ;
		\path[->] (0,0) edge node[very near end, below]{$2$} (1.5,0) ;
		\draw[dashed] (0,0) edge node[very near end, below]{$L_z$} (1.0, -1.0);
		\draw[dashed] (0,0) edge (-0.5, 0.5);
		\draw[dashed] (-0.02, 1.3) edge (0.75, -0.75);
		\path[->] (0.50, -0.40) edge (0.25, -0.55);
		\end{tikzpicture}
		\caption{$(N^*_s, A_s)$}
		\end{subfigure}
		~
\begin{subfigure}[t]{0.35\textwidth}
	\begin{tikzpicture}
	\path[->] (0,0) edge node[at end, left]{$1$} (0,1.5) ;
	\path[->] (0,0) edge node[very near end, below]{$2$} (1.5,0) ;
	\draw[dashed] (-0.5, 0) edge node[very near end, above]{} (1.75, 0);
	\path[->] (0.5,0) edge (0.5, -0.25);
	\end{tikzpicture}
	\caption{$(N_s, A_s)$}
\end{subfigure}
~
	\begin{subfigure}[t]{0.35\textwidth}
		\begin{tikzpicture}
		\path[->] (0,0) edge node[at end, left]{$1$} (0,1.5) ;
		\path[->] (0,0) edge node[very near end, below]{$2$} (1.5,0) ;
		\draw[dashed] (0, -0.5) edge node[very near end, above]{} (0, 1.75);
		\path[->] (0, 0.5) edge (-0.25, 0.5);
		\end{tikzpicture}
		\caption{$(N'_s, A_s)$}
	\end{subfigure}
	\caption{$L_1 < L_2 < L_z $}
	\label{ord3}
\end{figure}

Let $(N_1, A_1), (N_1', A_1), (N_2, A_2), (N_2', A_2) \ldots, (N_k, A_k), (A_k', A_k)$ be the ordering of atoms in the interval $\mathrm{MacP}(1,2,n)_{\leq (N,M)} \cup \{\hat{0}\}$. We will now verify the second condition in Definition ~\ref{rao}. 

Let $(W, T) \in \mathrm{MacP}(1,2,n)_{\leq (N, M)}$, and $(N^1, A^1), (N^2, A^2)$ atoms in $\mathrm{MacP}(1,2,n)_{\leq (N,M)} \cup \{\hat{0}\}$. Let $A^1 = (i_1, j_1)$ and $A^2 = (i_2, j_2)$. Suppose $(N^1, A^1)$ precedes $(N^2, A^2)$ in the ordering of atoms in $\mathrm{MacP}(1,2,n)$, and $(N^1, A^1), (N^2, A^2) < (W, T)$. 

We recall from Proposition \ref{rao} that  $(i, jk)$  denote a rank $2$ oriented matroid covering $(i,j)$, $(i,k)$, and $j$ is parallel/anti-parallel to $k$. Similarly, we have the notation $(ik, j)$ . We will then consider the following cases:

\textbf{Case 1:} Suppose $A^1 = A^2 = (i_2, j_2)$ so that $N^1 \neq N^2$. Let $w$ be a non-zero covector of $W$. We denote by $w^* = \max\{c\in \mathcal{V}^*(A^2)\setminus 0 : c \leq w\}$. Let $Z$ be a rank $1$ oriented matroid with covectors $\{0, \pm w^*\}$. Then $(Z, A^2) \leq (W, T)$ and covers both $(N^1, A^2)$ and $(N^2, A^2)$.

\textbf{Case 2:} Suppose $A^1 \neq A^2$ so that $A^1$ precedes $A^2$ in atom ordering in $\mathrm{MacP}(2,n)_{\leq M} \cup \{\hat{0}$. Suppose $i_1 = i_2$ so that $j_1 < j_2$. We obtained in Proposition ~\ref{rao} the rank 2 oriented matroid $L = (i_2, jj_2)$ such that $L \leq T$ and $L$ covers $A^2$ and $L' = (i_2, j)$ for some $j$ satisfying $j_1 \leq j < j_2$. 

Suppose $L_z \leq  i_2$ as in Figure \ref{ord1}. Let $w$ be a cocircuit of $L$ satisying $w(i_2) = 0$.  Let $w_1^* = \max\{c\in \mathcal{V}^*(L')\setminus 0 : c \leq w\}$. Let $Z, Z_1$ be rank $1$ oriented matroids with covectors $\{0, \pm w\} \; \mathrm{and} \; \{0, \pm w_1^*\}$ respectively. Then $(Z, L) \leq (W, T)$ and $(Z, L)$ covers $(Z_1, L')$ and $(N^2, A^2)$.

Suppose $ i_2 < L_z \leq j_2$ or $L_z \geq j_2$ . Let $w$ be a cocircuit of $L$ with $w(j) = w(j_2) = 0$. We obtain $Z$ and  $Z_1$ as described above, so that $(Z, L) \leq (W, T)$ and $(Z, L)$ covers $(Z_1, L')$ and $(N^2, A^2)$.

\textbf{Case 3:} Suppose $A^1 \neq  A^2$ and $i_1 < i_2$. As in the proof of Proposition ~\ref{rao}, we obtain $L = (ii_2 ,j_2) \leq T $ such that $L$ covers $A^2$ and $(i, j_2)$ for some $i_1 \leq i < i_2$. Let $L' = (i, j_2)$. The flags $(Z, L)$ and $(Z_1, L')$ can be obtain as in \textbf{Case 2} by considering the cases when $L_z \leq i$, $i < L_z \leq j_2$ or $L_z > j_2$.
\end{proof}

The following lemma also follows from the thinness of intervals of length $2$ in $\mathrm{MacP}(2, n+1)$, and the identification of an interval in  $\mathrm{MacP}(1,2,n)$ with an interval in $\mathrm{MacP}(2,n+1)$.

\begin{lemma}\label{thin1}
	Every interval of length $2$ in $\mathrm{MacP}(1,2,n)_{\leq (\pm z, M)} \cup \hat{0}$ is thin.
\end{lemma}

Using the above recursive atom ordering and Lemma ~\ref{thin1}, we have the following proposition. 
\begin{prop}
	Let $(N, M) \in \mbox{MacP}(1,2,n)$. The interval $\mathrm{MacP}(1,2,n)_{\leq (N, M)} \cup \{\hat{0}\}$ is an augmented face poset of a regular cell decomposition of a PL sphere.
\end{prop}
\begin{proof}
	The interval $\mathrm{MacP}(1,2,n)_{\leq (N, M)} \cup \{\hat{0}\}$ has a  recursive atom ordering by Proposition \ref{rec1} , and the interval is thin by Lemma \ref{thin1}. Hence by Theorem \ref{bjorner}, the interval $\mathrm{MacP}(1,2,n)_{\leq (N, M)} \cup \{\hat{0}\}$ is an augmented face poset of a regular cell decomposition of a PL sphere.
\end{proof}

\section{Topology of $\overline{\nu^{-1}(N,M)}$} \label{sec:top_2}
\begin{thm}\label{top1}
	Let $(N, M) \in \mbox{MacP}(1,2,n)$. Then the closure $\overline{\nu^{-1}(N,M)}$ is a topological manifold whose boundary is $\bigcup_{(N', M') < (N,M)} \nu^{-1}(N', M').$
\end{thm}

\begin{proof}
	Let $(Y,X) \in \partial \overline{\nu^{-1}(N,M)}$. The proof follows as in Theorem \ref{top}, with $B_M$ replaced with the ball
	$$ B_{(N,M)} =  \left \{(x_1, x_2, \ldots \Theta, \ldots , x_{n_1}) : 0 \leq x_1 \leq x_2 \cdots \leq \Theta \cdots \leq  x_{n_1} \leq\frac{\pi}{2}\right \} 
	 \times$$  $$\left\{(y_1, y_2, \ldots, y_{n_2}) : \frac{\pi}{2} \leq y_1 \leq y_2 \leq \cdots \leq y_{n_2}\leq\pi \right\} \times  \prod_{j=1}^{L_M-2} [\frac{r_j}{2}, r_j+1].$$ 
	 To obtain a map $T' : B_{(N, M)}/\sim  \; \rightarrow \overline{\nu^{-1}(N,M)}$. The image $T'(B_{(N, M)}/\sim)$ is the desired neighborhood of the point $(Y, X)$, and $T'(B_{(N,M)}/\sim)$ is homeomorphic to a closed ball.
\end{proof}

\begin{thm}\label{tops}
Let $(N, M) \in \mbox{MacP}(1,2,n)$. There is an homeomorphism from $\overline{\nu^{-1}(N, M)}$ to an $h(M) + h_M(N)$-dimensional closed ball.   
\end{thm}

\begin{proof}
	The proof follows as in Theorem \ref{cball}.
\end{proof}

\begin{proof}[Proof of Theorem \ref{main1}]
We conclude from Theorem \label{tops} that $\{\nu^{-1}(N, M) : (N, M) \in \mbox{MacP}(1,2,n) \}$ is a regular cell decomposition of $\mathrm{Gr}(1,2,\mathbb{R}^n)$.

Hence, $\|\mathrm{MacP}(1,2,n)\|$ is homeomorphic to $\mathrm{Gr}(1,2,\mathbb{R}^n)$.
\end{proof}

\section{Conclusion}
We have proven that the complexes $\|\mathrm{MacP}(2,n)\|$ and $\|\mathrm{MacP}(1,2,n)\|$ associated to combinatorial Grassmannians have the same heomorphism type as the Grassmannian manifolds $\mathrm{Gr}(2, \mathbb{R}^n)$ and $\mathrm{Gr}(1,2, \mathbb{R}^n)$ respectively. Our argument relies majorly on the realizability of rank $2$ oriented matroids, and the fact that stratas $\mu^{-1}(M)$ and $\nu^{-1}(\pm z, M)$ are homeomorphic to an open ball - the argument thus does not apply to oriented  matroids of ranks at least $3$.

\bibliography{MacP2}


\end{document}